\documentclass[12pt]{article}
\usepackage{times}
\usepackage[a4paper,text={128mm,185mm}, centering]{geometry}

\usepackage{changepage}
\usepackage{titlesec}

\titleformat{\section}[hang]%
{\bfseries\large}{\thesection.}{1ex}{}%

\titleformat{\subsection}[hang]%
{\bfseries}{\thesubsection}{1ex}{}%

\usepackage{amsthm}
\usepackage{fancyhdr}
\usepackage[utf8]{inputenc}
\usepackage{amsmath}
\usepackage{mathrsfs}
\usepackage{amssymb}
\usepackage{bm}
\usepackage{tikz-cd}
\usetikzlibrary{calc}
\usepackage{calc}
\usepackage{cite}
\usepackage{url}
\usepackage[symbol]{footmisc}
\usepackage{etoolbox}
\usepackage{hyperref}

\usepackage{newtxmath}
\usepackage[italic]{mathastext}

\tikzcdset{arrow style=math font}

\newtheorem{thm}{Theorem}[section]
\newtheorem{lemma}[thm]{Lemma}
\newtheorem{cor}[thm]{Corollary}
\newtheorem{prop}[thm]{Proposition}
\newtheorem*{thm*}{Theorem}
\newtheorem*{prop*}{Proposition}

\theoremstyle{remark}
\newtheorem{rk}[thm]{Remark}
\newtheorem{constr}[thm]{Construction}
\newtheorem{warn}[thm]{Warning}
\newtheorem{ex}[thm]{Example}
\newtheorem*{claim*}{Claim}

\AtBeginEnvironment{claim*}{\let\oldqed=\qedsymbol%
\renewcommand\qedsymbol{$\triangle$}}
\AtEndEnvironment{claim*}{\let\qedsymbol=\oldqed}

\theoremstyle{definition}
\newtheorem{defi}[thm]{Definition}

\numberwithin{equation}{section}

\newcommand{\Ho}{\textup{Ho}}
\newcommand{\nerve}{\textup{N}}
\newcommand{\h}{\textup{h}}
\newcommand{\cat}[1]{\textbf{\textup{#1}}}

\newcommand{\blank}{{\textup{--}}}

\newcommand{\Hom}{{\textup{Hom}}}
\newcommand{\id}{{\textup{id}}}

\newcommand{\im}{\mathop{\textup{im}}}

\newcommand{\colim}{\mathop{\textup{colim}}\nolimits}
\newcommand{\forget}{\mathop{\textup{forget}}\nolimits}
\newcommand{\Fun}{\textup{Fun}}

\newcommand{\supp}{\mathop{\textup{supp}}}

\newcommand{\core}{\mathop{\textup{core}}}

\newcommand{\Ob}{\mathop{\textup{Ob}}}
\newcommand{\zigzag}{\mathord{\text{\raisebox{.8pt}{$\scriptstyle\bullet$}\kern-4pt$\to$\kern-1pt\raisebox{.8pt}{$\scriptstyle\bullet$}\kern-1pt$\gets$\kern-4pt\raisebox{.8pt}{$\scriptstyle\bullet$}}}}

\newcommand{\Sd}{\mathop{\textup{Sd}}\nolimits}
\newcommand{\Ex}{\mathop{\textup{Ex}}\nolimits}
\newcommand{\hq}{\mathord{\textup{/\kern-2.5pt/}}}
\newcommand{\myh}{\mathord{\textup{`$\mskip-.1\thinmuskip h\mskip-.4\thinmuskip$'}}}

\newcommand{\triv}{\mathord{\textup{triv}}}

\let\phi=\varphi

\let\del=\partial

\def\twocell[#1]{\arrow[#1, dash, phantom, "\Rightarrow"{scale=1.125, yshift=-.4pt, description, allow upside down, sloped, inner sep=0pt}]}

\makeatletter
\def\mylabel#1#2{{\edef\@currentlabel{#2}\label{#1}}}
\makeatother

\title{\vskip 5pt  \bf  CATEGORICAL MODELS OF\\ UNSTABLE $\bm G$-GLOBAL HOMOTOPY THEORY}
\author{\itshape\bfseries {Tobias LENZ}}
\date{}

\begin{document}
\maketitle
\thispagestyle{plain}
\vskip 25pt
\begin{adjustwidth}{0.5cm}{0.5cm}
{\small
{\bf Abstract.}
    We prove that the category $\cat{$\bm G$-Cat}$ of small categories with $G$-action forms a model of unstable $G$-global homotopy theory for every discrete group $G$, generalizing Schwede's global model structure on $\cat{Cat}$. As a consequence, we prove that $\cat{$\bm G$-Cat}$ models proper $G$-equivariant homotopy theory not only when we test weak equivalences on fixed points, but also when we test them on categorical \emph{homotopy} fixed points.
}
\end{adjustwidth}

\section*{Introduction}
It is an observation going back to Quillen \cite[VI.3]{nerve-homotopy-category} that every topological space is weakly equivalent to the classifying space of a small category, and that in fact taking classifying spaces yields an equivalence between the homotopy category of small categories (formed with respect to those functors that induce homotopy equivalences of classifying spaces) and the usual unstable homotopy category. This comparison was later lifted to a model categorical statement by Thomason \cite{thomason-cat} who constructed a model structure on the category $\cat{Cat}$ of small categories with the above weak equivalences and proved that it is Quillen equivalent to the usual Kan-Quillen model structure on simplicial sets.

In more recent years, several generalizations and refinements of Thomason's and Quillen's results have been established:

In \cite{g-wit}, Bohmann, Mazur, Osorno, Ozornova, Ponto, and Yarnall constructed for any discrete group $G$ a model structure on the category $\cat{$\bm G$-Cat}$ of small $G$-categories in which a map is a weak equivalence if and only if it induces weak equivalences on all fixed points. Moreover, they proved that $\cat{$\bm G$-Cat}$ is Quillen equivalent to the usual $G$-equivariant model structure on $G$-simplicial sets, thereby establishing $\cat{$\bm G$-Cat}$ as a model of unstable \emph{$G$-equivariant homotopy theory}. This result was strengthened by May, Stephan, and Zakharevich \cite{equivariant-posets} who showed that already the full subcategory of $G$-posets models the same homotopy theory, generalizing a non-equivariant result due to Raptis \cite{posets}.

On the other hand, we can consider \emph{global homotopy theory} \cite{schwede-book} which, roughly speaking, studies equivariant phenomena that exist uniformly across suitable families of groups, like all finite groups or all compact Lie groups. In this setting, Schwede \cite{schwede-cat} refined Thomason's result by constructing the so-called \emph{global model structure} on $\cat{Cat}$ and proving that it is Quillen equivalent to the orbispace model of unstable global homotopy theory with respect to finite groups.

In the present paper, we generalize Schwede's result by establishing $\cat{$\bm G$-Cat}$ for every discrete group $G$ as a model of unstable \emph{$G$-global homotopy theory} in the sense of \cite[Chapter~1]{g-global}. $G$-global homotopy theory arises for example naturally in the study of global infinite loop spaces \cite[Chapter~2]{g-global} or in the form of various `Galois-global' phenomena \cite{schwede-galois}.

For every $G$, $G$-global homotopy theory admits a Bousfield localization to \emph{proper} $G$-equivariant homotopy theory---i.e.~equivariant homotopy theory where we only consider the fixed points for \emph{finite} subgroups---and we make the localization functor explicit for the model constructed in \cite{g-wit}. As a consequence of this comparison, we obtain a new model structure on $\cat{$\bm G$-Cat}$ that still models proper $G$-equivariant homotopy theory, but whose weak equivalences are now tested on \emph{categorical homotopy fixed points} (i.e. homotopy fixed points formed with respect to the underlying equivalences of categories) as opposed to ordinary fixed points. There is then a Quillen equivalence between this new model structure and the one of \cite{g-wit}, whose right adjoint is given by
\begin{equation}\tag{\ensuremath{*}}\label{eq:hfp-vs-fp}
\Fun(EG,\blank)\colon\cat{$\bm G$-Cat}_{\textup{homotopy fixed points}}\to\cat{$\bm G$-Cat}_{\textup{fixed points}}
\end{equation}
where $EG$ is the contractible groupoid with object set $G$, equipped with the evident $G$-action.

Our interest in the above model structure and the Quillen equivalence $(\ref{eq:hfp-vs-fp})$ comes from \emph{equivariant algebraic $K$-theory} as studied in \cite{guillou-may,merling}. Namely, as observed in \cite[§3]{gmmo}, $(\ref{eq:hfp-vs-fp})$ lifts to a functor from the category of small symmetric monoidal categories with $G$-action to the category of so-called \emph{genuine symmetric monoidal $G$-categories}. Equivariant algebraic $K$-theory in the sense of \cite{guillou-may} is defined in terms of the latter, and it is only this lift of $(\ref{eq:hfp-vs-fp})$ that allows to define the $G$-equivariant algebraic $K$-theory of a plain symmetric monoidal category with $G$-action.

In the sequel \cite{genuine-vs-naive}, we will prove that also this lift induces an equivalence of homotopy theories with respect to the above notions of weak equivalences; in particular, from the point of view of algebraic $K$-theory, there is no harm in just working with ordinary symmetric monoidal categories with $G$-action. While the argument we will give in \cite{genuine-vs-naive} will be formally mostly independent of the results of the present paper, the equivalence $(\ref{eq:hfp-vs-fp})$ provided much of the original motivation for \cite{genuine-vs-naive}. Moreover, the proof we give here requires much less machinery than its symmetric monoidal counterpart, and we think it is actually instructive to have a direct argument available in this case.

\subsection*{Organization}
In Section~\ref{sec:preliminaries} we recall some basic facts about Thomason's model structure on $\cat{Cat}$ as well as our reference model of unstable $G$-global homotopy theory in terms of simplicial sets equipped with an action of a specific simplicial monoid.

Section~\ref{sec:transfer} is devoted to establishing a general criterion for the existence of transferred model structures, which we then employ in Section~\ref{sec:monoid} to construct (under a mild technical assumption) a model structure on the category of small categories with the action of a given categorical monoid and to compare it to its simplicial counterpart, partially generalizing \cite{g-wit}. Using this, we establish a categorical analogue of our reference model of $G$-global homotopy theory and prove that these two models are Quillen equivalent.

Finally, we construct the desired $G$-global model structure on $\cat{$\bm G$-Cat}$ in Section~\ref{sec:g-cat} and compare it to our previous models of $G$-global homotopy theory as well as the proper $G$-equivariant model structure on $\cat{$\bm G$-Cat}$.

\subsection*{Acknowledgements}
I would like to thank Stefan Schwede for helpful comments on a previous version of parts of this article. I am moreover grateful to the anonymous referee for valuable feedback, in particular making me aware of nominal sets.

I would like to thank the Max Planck Institute for Mathematics in Bonn for their hospitality and financial support during the time the first version of this article was written. At that time, I moreover was an associate member of the Hausdorff Center for Mathematics, funded by the Deutsche Forschungsgemeinschaft (DFG, German Research Foundation) under Germany's Excellence Strategy (GZ 2047/1, project ID390685813).

\section{Preliminaries}\label{sec:preliminaries}
\subsection{\texorpdfstring{$\bm G$}{G}-global homotopy theory} \hspace{1sp}{\cite[Chapter~1]{g-global}} introduces several equivalent models of unstable $G$-glo\-bal homotopy theory and studies their relation to $G$-equivariant and global homotopy theory. Here we will recall one of these models, which is based on a specific monoid $\mathcal M$ that we call the \emph{universal finite group}:

\begin{defi}
We write $\omega=\{0,1,\dots\}$ for the set of natural numbers and we denote by $\mathcal M$ the monoid (under composition) of all injections $\omega\to\omega$.
\end{defi}

\begin{rk}
    In addition to their role in ($G$-)global homotopy theory \cite{g-global,schwede-k-theory}, which we will detail below, $\mathcal M$-actions have been studied in various places in the literature, for example in relation to the homotopy groups of symmetric spectra \cite{schwede-semistable} or in the study of $E_\infty$-monoids \cite{I-vs-M-1-cat}. In several of these applications, one imposes an additional \emph{tameness} condition on the $\mathcal M$-action, demanding that the action on any given element $x$ only depend on the values of an injection $u\in\mathcal M$ on a suitable \emph{finite} set $\supp(x)\subset\omega$. On the other hand, sets with an action of the maximal subgroup $\core\mathcal M$ (i.e.~the group of bijective self-maps of $\omega$) satisfying an analogous notion of tameness have been studied in logic and theoretical computer science under the name \emph{nominal sets} \cite{nominal-sets}, and together with the equivariant maps they form a topos, called the \emph{Schanuel topos}. It is not hard to prove, and also follows by combining \cite[Proposition~5.6]{I-vs-M-1-cat} with \cite[Theorem~6.8]{nominal-sets}, that a \emph{tame} $\mathcal M$-action is already uniquely determined by the action of the invertible elements, i.e.~the Schanuel topos is equivalent to the category of tame $\mathcal M$-sets via the forgetful functor.

    In contrast to that, we will work with general $\mathcal M$-actions throughout the present paper (which are \emph{not} determined by the action of the maximal subgroup), and we instead refer the reader e.g.~to \cite[Sections~1.3 and 2.1]{g-global} or \cite[Section~4]{genuine-vs-naive} for the role of tameness in ($G$-)global homotopy theory.
\end{rk}

\begin{defi}
A finite subgroup $H\subset\mathcal M$ is called \emph{universal} if $\omega$ with the restriction of the tautological $\mathcal M$-action is a {complete $H$-set universe}.
\end{defi}

Here we call a countable $H$-set $\mathcal U$ a \emph{complete $H$-set universe} if every other countable $H$-set embeds into $\mathcal U$ equivariantly.

It is in fact not hard to show that every finite group $H$ admits an injective homomorphism $i\colon H\to\mathcal M$ such that $i(H)$ is universal, and that any two such homomorphisms differ only by conjugation with an invertible element of $\mathcal M$ \cite[Lemma~1.2.8]{g-global}. In particular, if $X$ is any simplicial set with an $\mathcal M$-action, then we can associate to this an $H$-fixed point space for any abstract finite group $H$ by picking such a homomorphism $i\colon H\to\mathcal M$ and taking $i(H)$-fixed points. However, while this space is independent of the chosen homomorphism $i$ up to isomorphism, this isomorphism itself is \emph{not} canonical, even up to homotopy. One way to solve this is to pass to a certain simplicial extension of the monoid $\mathcal M$, which relies on the following construction:

\begin{constr}
Recall that the functor $\cat{Cat}\to\cat{Set}$ sending a small category to its set of objects admits a right adjoint $E$ (usually called the `chaotic' or `indiscrete' category functor). Explicitly, $EX$ is the category with object set $X$ and precisely one morphism $x\to y$ for any $x,y\in X$. Composition in $EX$ is defined in the unique possible way, and if $f\colon X\to Y$ is a map of sets, then $Ef$ is the unique functor which is given on objects by $f$.

Likewise, the functor $\cat{SSet}\to\cat{Set}$ sending a simplicial set $Y$ to its set of $0$-simplices admits a right adjoint $E$. Explicitly, $E$ is given on objects by $(EX)_n=\prod_{i=0}^n X\cong\textup{maps}(\{0,\dots,n\},X)$ with the evident functoriality in $n$ and $X$.

Note that the nerve of the category $EX$ is indeed canonically isomorphic to the simplicial set of the same name, justifying that we won't distinguish between them notationally.
\end{constr}

As $E$ is a right adjoint, it in particular preserves products so that the category or simplicial set $E\mathcal M$ inherits a natural monoid structure from $\mathcal M$.

\begin{warn}
    If $G$ is a discrete group acting on a set $X$, one can form the \emph{translation category} $X\mskip.4\thinmuskip/\mskip-.4\thinmuskip/\mskip.4\thinmuskip G$, whose objects are given by the set $X$ and with $\Hom(x,y)=\{g\in G:g.x=y\}$ for any $x,y\in X$. If $X=G$ with its usual action, this agrees with the indiscrete category $EG$, see e.g.~\cite[Proposition~1.8]{categorical-classifying}, and accordingly some sources like \cite[Section~10]{mandell-translation} refer to $EG$ as the translation category (while still using the above notation). Beware however that while the construction of the translation category makes perfect sense for any monoid action, the result for the monoid $\mathcal M$ would be different from the indiscrete category $E\mathcal M$ (in particular, the translation category $\mathcal M\mskip.4\thinmuskip/\mskip-.4\thinmuskip/\mathcal M$ contains nontrivial endomorphisms).
\end{warn}

\begin{defi}
We write $\cat{$\bm{E\mathcal M}$-$\bm G$-SSet}$ for the category whose objects are the simplicial sets equipped with an action of the simplicial monoid $E\mathcal M\times G$, and whose morphisms are the $(E\mathcal M\times G)$-equivariant maps. A map $f\colon X\to Y$ in $\cat{$\bm{E\mathcal M}$-$\bm G$-SSet}$ is called a \emph{$G$-global weak equivalence} if $f^\phi$ is a weak equivalence for every universal subgroup $H\subset\mathcal M$ and every homomorphism $\phi\colon H\to G$; here we write $(\blank)^\phi$ for the fixed points with respect to the \emph{graph subgroup} $\Gamma_{H,\phi}\mathrel{:=}\{(h,\phi(h)) : h\in H\}\subset\mathcal M\times G$.
\end{defi}

Next, we want to recall the \emph{$G$-global model structure} on $\cat{$\bm{E\mathcal M}$-$\bm G$-SSet}$. This is actually just a particular instance of the following proposition, which generalizes the usual equivariant model structures for group actions, see e.g.~\cite[Example~2.14]{cellular}, to monoid actions:

\begin{prop}\label{prop:equiv-model-structure}
Let $M$ be a simplicial monoid and let $\mathcal F$ be a collection of finite subgroups of $M_0$. Then there exists a unique model structure on $\cat{$\bm M$-SSet}$ in which a map $f$ is a weak equivalence or fibration if and only if $f^H$ is a weak equivalence or fibration, respectively, in the usual Kan-Quillen model structure on $\cat{SSet}$ for each $H\in\mathcal F$. We will refer to this as the \emph{$\mathcal F$-model structure} and to its weak equivalences as the \emph{$\mathcal F$-weak equivalences}.

The $\mathcal F$-model structure is combinatorial with generating cofibrations
\begin{equation*}
\{M/H\times\del\Delta^n\hookrightarrow M/H\times\Delta^n:H\in\mathcal F,n\ge0\}
\end{equation*}
and generating acyclic cofibrations
\begin{equation*}
\{M/H\times\Lambda_k^n\hookrightarrow M/H\times\Delta^n:H\in\mathcal F,0\le k\le n\}.
\end{equation*}
Moreover it is simplicial (for the obvious enrichment), proper, and a commutative square is a homotopy pushout or pullback if and only if the induced square on $H$-fixed points is a homotopy pushout or pullback, respectively, in $\cat{SSet}$ for every $H\in\mathcal F$. Pushouts along underlying cofibrations are homotopy pushouts.

Finally, the $\mathcal F$-weak equivalences are stable under filtered colimits.
\begin{proof}
\cite[Proposition~1.1.2]{g-global} shows all of these except for the characterizations of homotopy pushouts and pullbacks. The statement about homotopy pullbacks is obvious, while the ones about homotopy pushouts are instances of \cite[Proposition~1.1.6 and Lemma~1.1.14]{g-global}.
\end{proof}
\end{prop}

Specializing to our situation we get, also see~\cite[Corollary~1.2.34]{g-global}:

\begin{cor}\label{cor:EM-G-SSet-model-structure}
There is a unique model structure on $\cat{$\bm{E\mathcal M}$-$\bm G$-SSet}$ in which a map $f$ is a weak equivalence or fibration if and only if $f^\phi$ is a weak equivalence or fibration, respectively, in the usual Kan-Quillen model structure on $\cat{SSet}$.\qed
\end{cor}

In particular, the weak equivalences of this model structure are precisely the $G$-global weak equivalences, and accordingly we refer to this as the \emph{$G$-global model structure}. Proposition~\ref{prop:equiv-model-structure} also provides us with explicit sets of generating (acyclic) cofibrations: writing $E\mathcal M\times_\phi G$ for $(E\mathcal M\times G)/\Gamma_{H,\phi}$ these are given by
\begin{align*}
\{E\mathcal M\times_\phi G\times(\del\Delta^n\hookrightarrow\Delta^n) &: n\ge 0,\text{ $H\subset\mathcal M$ universal},\phi\colon H\to G\}\\
\{E\mathcal M\times_\phi G\times(\Lambda^n_k\hookrightarrow\Delta^n) &: 0\le k\le n,\text{ $H\subset\mathcal M$ universal}, \phi\colon H\to G\}.
\end{align*}

Finally, we come to the relation between $G$-global and \emph{proper $G$-equi\-vari\-ant homotopy theory}, i.e.~$G$-equivariant homotopy theory with respect to the collection of all \emph{finite} subgroups $H\subset G$.

\begin{thm}\label{thm:G-global-vs-G-equiv-sset}
The functor $\triv_{E\mathcal M}\colon\cat{$\bm G$-SSet}_{\textup{proper}}\to\cat{$\bm{E\mathcal M}$-$\bm G$-SSet}_{\textup{$G$-global}}$ equipping a $G$-simplicial set with the trivial $E\mathcal M$-action is homotopical. The induced functor $\Ho(\triv_{E\mathcal M})\colon\Ho(\cat{$\bm G$-SSet})\to\Ho(\cat{$\bm{E\mathcal M}$-$\bm G$-SSet})$ on homotopy categories fits into a sequence of four adjoints suggestively denoted by
\begin{equation*}
\cat{L}(\blank/E\mathcal M)\dashv\Ho(\triv_{E\mathcal M})\dashv (\blank)^{\cat{R}E\mathcal M}\dashv\mathcal R.
\end{equation*}
Moreover, $(\blank)^{\cat{R}E\mathcal M}$ is a (Bousfield) localization at the $\mathcal E$-weak equivalences, where $\mathcal E$ denotes the collection of all subgroups $\Gamma_{H,\phi}\subset\mathcal M\times G$ with universal $H\subset\mathcal M$ for which $\phi$ is \emph{injective}.

In particular, $\triv_{E\mathcal M}:\cat{$\bm G$-SSet}_{\textup{proper}}\to\cat{$\bm{E\mathcal M}$-$\bm G$-SSet}_{\textup{$\mathcal E$-w.e.}}$ descends to an equivalence of homotopy categories.
\begin{proof}
See \cite[Theorem~1.2.92]{g-global}.
\end{proof}
\end{thm}

\begin{rk}
    In fact, \emph{loc.~cit.}\ establishes the above result on the level of $\infty$-categorical localizations. For simplicity, we will stick to the formulation in terms of classical homotopy categories in the present paper.
\end{rk}

\subsection{The Thomason model structure} We close this section by recalling Thomason's model structure on $\cat{Cat}$ that models the ordinary homotopy theory of spaces. While the usual nerve functor $\nerve$ induces an equivalence of homotopy categories by \cite[Corollaire~3.3.1]{nerve-homotopy-category}, it can't be part of a Quillen equivalence to the Kan-Quillen model structure as its left adjoint $\h$ (sending a simplicial set to its \emph{homotopy category}) is not homotopically well-behaved. Thomason's crucial insight was that we can avoid this issue by using Kan's $\Sd\dashv\Ex$-adjunction \cite[§7]{sd-ex} to replace the nerve by a weakly equivalent functor:

\begin{thm}[Thomason]\label{thm:thomason}
There is a unique model structure on $\cat{Cat}$ in which a functor $f\colon C\to D$ is a weak equivalence if and only if $\nerve(f)$ is a weak homotopy equivalence of simplicial sets and a fibration if and only if $\Ex^2\nerve(f)$ is a Kan fibration. This model structure is combinatorial with generating cofibrations
\begin{equation*}
\{\h\Sd^2\del\Delta^n\hookrightarrow\h\Sd^2\Delta^n\colon n\ge 0\}
\end{equation*}
and generating acyclic cofibrations
\begin{equation*}
\{\h\Sd^2\Lambda_k^n\hookrightarrow\h\Sd^2\Delta^n\colon 0\le k\le n\}.
\end{equation*}
Moreover, with respect to this model structure the adjunction
\begin{equation}\label{eq:thomason-adjunction}
\h\Sd^2\colon\cat{SSet}_{\textup{Kan-Quillen}}\rightleftarrows\cat{Cat} :\Ex^2\nerve
\end{equation}
is a Quillen equivalence.
\begin{proof}
The existence of the model structure together with the above choices of generating (acyclic) cofibrations is \cite[Theorem 4.9]{thomason-cat}; as $\cat{Cat}$ is locally presentable, this is then a combinatorial model structure.

It is obvious that $(\ref{eq:thomason-adjunction})$ is a Quillen adjunction. Moreover, the right adjoint is homotopical as $\Ex$ is weakly equivalent to the identity functor \cite[Lemma~7.4]{sd-ex}, while the left adjoint is so by Ken Brown's Lemma. Thus, for $(\ref{eq:thomason-adjunction})$ to be a Quillen equivalence one has to show that the ordinary unit and counit are weak equivalences, for which Thomason refers to Fritsch and Latch \cite[Example~4.12-(v)]{fritsch-latch}.
\end{proof}
\end{thm}

Thomason's proof of the above theorem crucially relies on a careful analysis of the (generating) cofibrations. As we will need some of their properties later, we briefly recall them here for easy reference.

\begin{defi}\label{defi:ordinary-dwyer}
A sieve $i\colon C\to D$ is called a \emph{Dwyer map} if it can be factored as $i=jf$ such that the following holds:
\begin{enumerate}
\item $j$ is a cosieve.
\item $f$ admits a right adjoint.
\end{enumerate}
\end{defi}

\begin{prop}\label{prop:h-sd-cofib}
Let $i\colon K\hookrightarrow L$ be a cofibration of simplicial sets. Then $\h\Sd^2(i)$ is a Dwyer map.
\begin{proof}
See \cite[Proposition~4.2]{thomason-cat}.
\end{proof}
\end{prop}

Dwyer maps are extremely useful since pushouts along them admit a very explicit description, which we will recall later in Construction~\ref{constr:pushout-Dwyer}. For now we only record one important consequence of this:

\begin{prop}\label{prop:pushout-dwyer-nerve}
Let
\begin{equation*}
\begin{tikzcd}
A\arrow[d] \arrow[r, "i"] & B\arrow[d]\\
C\arrow[r] & D
\end{tikzcd}
\end{equation*}
be a pushout in $\cat{Cat}$ such that $i$ is a Dwyer map. Then the induced map
\begin{equation*}
\nerve B\amalg_{\nerve A}\nerve C\to\nerve D
\end{equation*}
is a weak homotopy equivalence.
\begin{proof}
This is \cite[Proposition~4.3]{thomason-cat}.
\end{proof}
\end{prop}

Finally, let us state Schwede's global refinement of the Thomason model structure, which we will later generalize to the $G$-global setting:

\begin{thm}[Schwede]\label{thm:thomason-schwede}
There is a unique model structure on $\cat{Cat}$ in which a functor $f\colon C\to D$ is a weak equivalence or fibration if and only if $\Fun(BH,f)$ is a weak equivalence or fibration, respectively, in the Thomason model structure for every finite group $H$.

We call this the \emph{global model structure} on $\cat{Cat}$. It is proper and combinatorial with generating cofibrations
\begin{equation*}
I=\{BH\times\h\Sd^2\del\Delta^n\hookrightarrow BH\times\h\Sd^2\Delta^n : n\ge 0, \text{$H$ a finite group}\}
\end{equation*}
and generating acyclic cofibrations
\begin{equation*}
J=\{BH\times\h\Sd^2\Lambda_k^n\hookrightarrow BH\times\h\Sd^2\Delta^n : 0\le k\le n, \text{$H$ a finite group}\}.
\end{equation*}
\end{thm}

Strictly speaking, $I$ and $J$ are not sets as there are too many finite groups. However, this can be easily cured by restricting to a system of representatives of isomorphism classes of finite groups, which we will tacitly assume below.

\begin{proof}
Specializing \cite[Theorem~1.12]{schwede-cat} to the collection $\{BH: H\text{ finite}\allowbreak\hskip\spaceskip\text{group}\}$ shows that this model structure exists and is proper, also see \cite[Theorem~3.3]{schwede-cat}. Moreover, Schwede's proof explicitly identifies $I$ and $J$ as set of generating cofibrations and generating acyclic cofibrations, respectively.
\end{proof}

The following lemma is crucial to Schwede's proof of the above theorem and will also be instrumental later in establishing our $G$-global generalization:

\begin{lemma}\label{lemma:pushout-dwyer-fun}
Let $X$ be a small category such that for every $x,y\in X$ there exists both a morphism $x\to y$ as well as $y\to x$. Let moreover
\begin{equation*}
\begin{tikzcd}
A\arrow[r, "i"]\arrow[d] & B\arrow[d]\\
C\arrow[r] & D
\end{tikzcd}
\end{equation*}
be a pushout in $\cat{Cat}$ where $i$ is a Dwyer map. Then also the induced square
\begin{equation*}
\begin{tikzcd}[column sep=large]
\Fun(X,A)\arrow[r, "{\Fun(X,i)}"]\arrow[d] & \Fun(X,B)\arrow[d]\\
\Fun(X,C)\arrow[r] & \Fun(X,D)
\end{tikzcd}
\end{equation*}
is a pushout.
\begin{proof}
This is the first half of \cite[Theorem~1.5]{schwede-cat}.
\end{proof}
\end{lemma}

\section{Transferring model structures}\label{sec:transfer}
Just as the usual equivariant model structures on $\cat{$\bm G$-Cat}$ or $\cat{$\bm G$-SSet}$, the $G$-global model structures we discuss in this paper will be obtained as \emph{transferred model structures}:

\begin{defi}
Let $\mathscr C$ be a model category, let $\mathscr D$ be a complete and cocomplete category, and let
\begin{equation*}
F\colon\mathscr C\rightleftarrows\mathscr D :\!U
\end{equation*}
be an adjunction. The \emph{model structure transferred along $F\dashv U$} is the (unique if it exists) model structure on $\mathscr D$ in which a morphism $f$ is a weak equivalence or fibration if and only if $Uf$ is a weak equivalence or fibration, respectively, in $\mathscr C$.
\end{defi}

We now give a criterion for the existence of transferred model structures that we will use for all our constructions later.

\begin{prop}\label{prop:transfer-criterion}
Let $\mathscr C$ be a left proper cofibrantly generated model category such that filtered colimits in $\mathscr C$ are homotopical, and let $I,J$ be sets of generating (acyclic) cofibrations. Moreover, let $\mathscr D$ be a locally presentable category together with an adjunction $F\colon\mathscr C\rightleftarrows\mathscr D :\!U$, and assume the following:
\begin{enumerate}
\item For each $j\in J$, the map $UFj$ is a weak equivalence.\label{item:tp-weak-equivalence}
\item $U$ sends any pushout square
\begin{equation*}
\begin{tikzcd}
FA\arrow[d]\arrow[r, "Fi"] & FB\arrow[d]\\
C \arrow[r] & D
\end{tikzcd}
\end{equation*}
in $\mathscr D$, where $i\in I$ is a generating cofibration of $\mathscr C$, to a homotopy pushout in $\mathscr C$.\label{item:tp-pushout}
\item $U$ preserves filtered colimits up to weak equivalence, i.e.~for each filtered poset $P$ and each diagram $X_\bullet\colon P\to\mathscr D$ the natural
comparison map $\colim_P(U\circ X_\bullet)\to U(\colim_PX_\bullet)$ is a weak equivalence.\label{item:tp-filtered}
\end{enumerate}
Then the transferred model structure on $\mathscr D$ exists and it is combinatorial with set of generating cofibrations $FI$ and set of generating acyclic cofibrations $FJ$. This model structure is left proper with homotopy pushouts created by $U$, and filtered colimits in $\mathscr D$ are homotopical; if $\mathscr C$ is right proper, then so is $\mathscr D$, and $U$ also creates homotopy pullbacks.

Moreover, in the presence of $(\ref{item:tp-pushout})$ and $(\ref{item:tp-filtered})$ the first condition is implied by
\begin{enumerate}
    \item[(\ref*{item:tp-weak-equivalence}$'$)]\mylabel{item:tp-unit-we}{1'} $J$ consists of maps between cofibrant objects. Moreover, the unit $\eta_\varnothing$ is a weak equivalence, and for each generating cofibration $(X\to Y)\in I$ both $\eta_X$ and $\eta_Y$ are weak equivalences.
\end{enumerate}
Finally, under this stronger assumption the adjunction $F\dashv U$ is a Quillen equivalence.
\begin{proof}
Assume first that $(\ref{item:tp-weak-equivalence})$--$(\ref{item:tp-filtered})$ hold. To see that the transferred model structure exists and is cofibrantly generated by $FI$ and $FJ$ it suffices to verify the assumptions of the usual transfer criterion for cofibrantly generated model categories \cite[Theorem~11.3.2]{hirschhorn}. As $\mathscr D$ is locally presentable, the smallness assumption is automatically satisfied, so we only have to show that relative $FJ$-cell complexes are weak equivalences. By Condition~$(\ref{item:tp-filtered})$ it is then enough to verify that pushouts of maps of the form $Fj$ with $j\in J$ are weak equivalences in $\mathscr D$, i.e.~sent under $U$ to weak equivalences in $\mathscr C$.

Let us consider the class $\mathscr H$ of those maps $i'\colon A'\to B'$ in $\mathscr D$ such that $U$ sends pushouts along them to homotopy pushouts, i.e.~those maps such that the analogue of Condition~($\ref{item:tp-pushout}$) holds for them. \cite[Proposition~A.2.7]{g-global} then shows that $\mathscr H$ is closed under pushouts, transfinite compositions, and retracts.

As $F$ preserves pushouts, transfinite compositions, and retracts (being a left adjoint functor), it follows that also $F^{-1}(\mathscr H)$ is closed under all of these. As it contains all $i\in I$ by assumption, it follows by the characterizations of cofibrations in a cofibrantly generated model category that $F^{-1}(\mathscr H)$ contains all cofibrations of $\mathscr C$; in particular it contains $J$. Hence if $(j\colon A\to B)\in J$ is a generating acyclic cofibration and we have any pushout square
\begin{equation}\label{diag:push-out-gen-acyclic}
\begin{tikzcd}
FA\arrow[d]\arrow[r, "Fj"] & FB\arrow[d]\\
C \arrow[r, "k"'] & D,
\end{tikzcd}
\end{equation}
then applying $U$ to this yields a homotopy pushout in $\mathscr D$. But $UFj$ is a weak equivalence by Condition~$(\ref{item:tp-weak-equivalence})$. It follows that $Uk$ is a weak equivalence, and hence by definition so is $k$. Altogether, we conclude that the transferred model structure exists and is cofibrantly generated by $FI$ and $FJ$ (hence combinatorial).

But with this established we conclude by the same argument (this time applied in $\mathscr D$) from the closure properties of $\mathscr H$ that $U$ sends pushouts along cofibrations in $\mathscr D$ to homotopy pushouts. Thus, \cite[Lemma~A.2.15]{g-global} shows that $\mathscr D$ is left proper with homotopy pushouts created by $U$. The statements about filtered colimits and homotopy pullbacks are trivial, finishing the proof of the first half of the proposition.

Now assume that $(\ref{item:tp-unit-we})$, $(\ref{item:tp-pushout})$, and $(\ref{item:tp-filtered})$ hold. We first observe:

\begin{claim*}
The unit $\eta_X$ is a weak equivalence for each cofibrant $X\in\mathscr C$.
\begin{proof}
This is a standard cell induction argument. By Quillen's Retract Argument, any cofibrant object is a retract of an $I$-cell complex; as weak equivalences are closed under retracts, it therefore suffices to prove the claim for every $I$-cell complex $X$.

To this end, we write $X$ as a transfinite composition $\varnothing=X_0\to X_1\to\cdots X_\alpha=X$ of pushouts of maps in $I$ for some ordinal $\alpha$. We will now prove by transfinite induction that $\eta_{X_\beta}$ is a weak equivalence for every $\beta\le\alpha$.

For $\beta=0$ this is part of Condition~$(\ref{item:tp-unit-we})$. If $\beta=\gamma+1$ is a successor ordinal, then we exhibit $X_\gamma\to X_\beta$ as a pushout of some generating cofibration $i\colon A\to B$ and consider the induced commutative cube
\begin{equation}\label{diag:tp-cube}
\begin{tikzcd}[column sep=small, row sep=small]
 & UFA \arrow[rr, "UFi"]\arrow[dd] && UFB\arrow[dd]\\
A\arrow[ur] \arrow[rr, "\qquad i", crossing over]\arrow[dd] && B\arrow[ur]\\
 & UFX_\gamma\arrow[rr] && UFX_\beta\\
X_\gamma\arrow[rr]\arrow[ur] && X_\beta\arrow[from=uu, crossing over]\arrow[ur]
\end{tikzcd}
\end{equation}
where all front-to-back maps are given by $\eta$. The front square is a homotopy pushout as $\mathscr C$ is left proper, and so is the back square by Condition~$(\ref{item:tp-pushout})$ and since $F$ preserves pushouts.

In $(\ref{diag:tp-cube})$, the upper front-to-back maps are weak equivalences by Condition~$(\ref{item:tp-unit-we})$, and so is the lower left one by the induction hypothesis. Thus, also $\eta_{X_\beta}$ is a weak equivalence as desired.

Finally, if $\beta$ is a limit ordinal, then we consider the commutative square
\begin{equation*}
\begin{tikzcd}
\colim\nolimits_{\gamma<\beta} X_\gamma\arrow[r]\arrow[d, "{\colim_{\gamma<\beta}\eta}"'] & X_\beta\arrow[d, "\eta"]\\
\colim\nolimits_{\gamma<\beta} UFX_\gamma\arrow[r] & UFX_\beta
\end{tikzcd}
\end{equation*}
where the horizontal maps are induced by the structure maps $X_\gamma\to X_\beta$; in particular, the upper map is an isomorphism and the lower map is a weak equivalence by Condition~$(\ref{item:tp-filtered})$. On the other hand, the left hand map is a filtered colimit of weak equivalences by the induction hypothesis, hence a weak equivalence by assumption on $\mathscr C$. Thus, also the right hand vertical map is a weak equivalence by $2$-out-of-$3$, which completes the proof of the claim.
\end{proof}
\end{claim*}

If now $j\colon X\to Y$ is one of the chosen generating acyclic cofibrations, then $X$ and $Y$ are cofibrant by assumption, so $\eta_X$ and $\eta_Y$ are weak equivalences by the above. Thus, also $UFj$ is a weak equivalence by $2$-out-of-$3$, proving $(\ref{item:tp-weak-equivalence})$, and in particular supplying the desired model structure. To show that $F\dashv U$ is a Quillen equivalence, we observe that $U$ creates weak equivalences by definition, so that it suffices that $\eta\colon X\to UFX$ is a weak equivalence for each cofibrant $X$, which was verified above.
\end{proof}
\end{prop}

\section{Categories with monoid actions}\label{sec:monoid}
In this section we want to prove the analogue of Proposition~\ref{prop:equiv-model-structure} for suitable strict monoids in $\cat{Cat}$ as well as a comparison between the simplicial and categorical approaches.

\subsection{Equivariant Dwyer maps for groups}
Dwyer maps are central to Thomason's treatment of his model structure on $\cat{Cat}$, and it should come as no surprise that we will need an equivariant version of this. For this it will be convenient to consider the case of ordinary groups first (as some arguments will only work in this setting), so let us fix a not necessarily finite discrete group $G$.

\begin{defi}\label{defi:dwyer-group}
A $G$-equivariant functor $i\colon C\to D$ of small $G$-categories is called a \emph{$G$-equivariant Dwyer map} if it is a sieve and it admits a factorization $i=jf$ into $G$-equivariant functors $f\colon C\to X$, $j\colon X\to D$ with the following properties:
\begin{enumerate}
\item $j$ is a cosieve.
\item $f$ admits a $G$-equivariant right adjoint, i.e.~there exists a $G$-equivariant functor $r\colon X\to C$ together with $G$-equivariant natural transformations $\eta\colon\id\Rightarrow ri$ and $\epsilon\colon ir\Rightarrow\id$ satisfying the usual triangle identities.\label{item:Dwyer-existence-adjoint}
\end{enumerate}
\end{defi}

\begin{rk}\label{rk:ordinary-dwyer}
For $G=1$ the above is equivalent to $i$ being an ordinary (i.e.~non-equivariant) Dwyer map, see Definition~\ref{defi:ordinary-dwyer}. Conversely, if $i$ is a non-equivariant Dwyer map, then it is a $G$-equivariant Dwyer map with respect to the trivial $G$-actions on source and target for any discrete group $G$.
\end{rk}

\begin{rk}\label{rk:nice-retraction}
May, Stephan, and Zakharevich \cite[Definition~4.1]{equivariant-posets} additionally require the map $\eta$ in $(\ref{item:Dwyer-existence-adjoint})$ to be the identity transformation; however, this can always be arranged, so that the above definition agrees with their notion of a \emph{Dwyer $G$-map}:

The first condition guarantees that $f$ is again fully faithful. It follows then formally that for any right adjoint $\tilde r$ the unit transformation $\tilde\eta\colon\id\Rightarrow\tilde rf$ is a natural isomorphism. As $f$ is obviously injective on objects, it is well-known in the non-equivariant setting that we may massage $\tilde r$ to another right adjoint $r$ of $f$ such that the unit $\eta$ is actually the identity.

The same proof works in the equivariant setting, but I do not know a reference for this. As this extra condition will become relevant later, let me briefly sketch the argument. We first define $r\colon X\to C$ on objects via
\begin{equation*}
r(x)=\begin{cases}
c & \text{if }x=f(c)\\
\tilde r(x) & \text{if }x\notin\im f.
\end{cases}
\end{equation*}
This is well-defined as $f$ is injective on objects, and it is $G$-equivariant because $\im f$ and hence also $(\im f)^c$ are closed under the action of $G$; note that this would break down for general monoids. We now define for each $x\in X$ an isomorphism $\phi_x\colon r(x)\to \tilde r(x)$ as follows: if $x=f(c)$, then $\phi_x$ is the unit $\tilde\eta_c\colon r(x)=c\to \tilde rf(c)=\tilde r(x)$; otherwise, $\phi_x$ is the identity. It is then obvious that this is again compatible with the $G$-action in the sense that $\phi_{g.x}=g.\phi_x$.

There is a unique way to extend $r$ to a functor in such a way that $\phi$ becomes a natural isomorphism $r\cong\tilde r$, namely $r(\alpha\colon x\to y)=\phi_y^{-1}\tilde r(\alpha)\phi_x$. It follows then from the above compatibility of $\phi$ with the $G$-action that $r$ is again $G$-equivariant and that $\phi$ is a $G$-equviariant isomorphism. One then easily checks that $\eta\mathrel{:=}\id=\phi^{-1}f\circ\tilde\eta\colon\id\Rightarrow rf$ and $\epsilon\mathrel{:=}\tilde\epsilon\circ f\phi\colon fr\Rightarrow\id$ exhibit $r$ as $G$-equivariant right adjoint of $f$ as desired.
\end{rk}

Below we will need the following closure properties of $G$-equivariant Dwyer maps, the first one of which can also be found (without proof) as \cite[Lemma~4.2]{equivariant-posets}.

\begin{lemma}\label{lemma:Dwyer-closure}
Let $i\colon C\to D$ be a $G$-equivariant Dwyer map.
\begin{enumerate}
\item\label{item:Dwyer-fixed-points} Let $H\subset G$ be a subgroup. Then $i^H\colon C^H\to D^H$ is a Dwyer map.
\item\label{item:Dwyer-product} Let $S$ be any small $G$-category. Then $S\times i\colon S\times C\to S\times D$ is a $G$-equivariant Dwyer map.
\item\label{item:Dwyer-fun} Let $T$ be a small right $G$-category. Then also $\Fun(T,i)\colon\Fun(T,C)\to\Fun(T,D)$ is a $G$-equivariant Dwyer map.
\end{enumerate}
\begin{proof}
All of these follow the same pattern, so we will only prove the first statement. We pick a factorization $i=jf$ as above and a $G$-equivariant adjunction $f\dashv r$ with unit $\eta$ and counit $\epsilon$. It is then easy to check that $i^H$ is again a sieve and that $j^H$ is a cosieve. Moreover, $r^H$ is right adjoint to $f^H$ with unit $\eta^H$ and counit $\epsilon^H$, so $i^H=j^Hf^H$ is the desired factorization.
\end{proof}
\end{lemma}

In general, pushouts in $\cat{Cat}$ (and hence also in $\cat{$\bm G$-Cat}$) are very difficult to describe on the level of morphisms. One advantage of ordinary Dwyer maps is that one can give an explicit and tractable description of pushouts along them, see \cite[Construction~1.2]{schwede-cat} which generalizes \cite[Lemma~2.5]{g-wit}:

\begin{constr}\label{constr:pushout-Dwyer}
Let
\begin{equation*}
\begin{tikzcd}
A \arrow[r, "i"]\arrow[d, "c"'] & B\\
C
\end{tikzcd}
\end{equation*}
be a diagram in $\cat{Cat}$ and assume that $i$ is a Dwyer map. Fix a factorization $i=kf$ as in the definition of a Dwyer map and an adjunction $f\dashv r$ such that the unit is the identity; in particular the counit $\epsilon$ satisfies $\epsilon f=\id$ by the triangle identities. For simplicity of notation we assume further that $i$ and $k$ are honest inclusions of subcategories; we write $X$ for the source of $k$. Finally, let us write $V$ for the complement of $\Ob A$ in $\Ob B$.

We now define a category $D$ as follows: the objects of $D$ are given by the disjoint union $\Ob C\amalg V$ and the morphism set between $x,y\in D$ is defined as
\begin{equation}\label{eq:hom-pushouts}
\Hom_D(x,y)=\begin{cases}
\Hom_C(x,y) & \text{if }x,y\in C\\
\Hom_C(x, cr(y)) & \text{if }x\in C\text{ and }y\in V\cap X\\
\Hom_B(x, y) & \text{if }x,y\in V\\
\varnothing & \text{otherwise}.
\end{cases}
\end{equation}
Compositions are in such a way that the obvious maps $B\to D$ and $C\to D$ are actual functors; moreover, if $x,y\in C$ and $z\in V\cap X$ then the composition
\begin{equation*}
x\xrightarrow{\alpha} y\xrightarrow{\beta} z
\end{equation*}
in $D$, where $\alpha$ is a morphism $x\to y$ in $C$ and $\beta$ is a morphism $y\to cr(z)$ in $C$, is defined as the composition $\beta\circ\alpha$ in $C$. On the other hand, if $x\in C$, $y,z\in V\cap X$, then the composition $\beta\circ\alpha$ in $D$, where now $\alpha\colon x\to cr(y)$ is a morphism in $C$ and $\beta\colon y\to z$ is a morphism in $V\cap X\subset B$, is defined as the composition $cr(\beta)\circ\alpha$ in $C$.

We have a functor $j\colon C\to D$ via the inclusion of $C$. Moreover, we define $d\colon B\to D$ as follows: on $V\subset B$ the functor $d$ is just given by the inclusion and on $A=V^c$ via $c$. Finally, if $\beta\colon a\to x$ is a morphism in $B$, where $a\in A$ and $x\in X\cap V$, then
\begin{equation*}
d(\beta)=c(r(\beta))\colon \underbrace{cr(a)}_{\makebox[0pt]{$\scriptstyle{}=c(a)=d(a)$}}\to cr(x)\in \Hom_C(c(a), cr(x))=\Hom_D(d(a), d(x)).
\end{equation*}
We remark that this indeed a complete case distinction as $X\subset B$ is a cosieve (so any morphism starting in $A\subset X$ has to end in $X=(X\cap V)\cup A$) and $A\subset B$ is a sieve (so any arrow ending in $A$ also has to start in $A$).

We omit the verification that $D$ is a category and that these are well-defined functors exhibiting $D$ as pushout (which uses that $\epsilon f=\id$), and instead refer the curious reader to \cite[Construction~1.2]{schwede-cat}.
\end{constr}

\begin{constr}\label{constr:pushout-G-Dwyer}
Now assume that $A,B,C$ are $G$-categories, $i$ is a $G$-equi\-vari\-ant Dwyer map, and $c$ is any equivariant functor. By Remark~\ref{rk:nice-retraction} we may choose the $G$-equivariant adjunction $f\dashv r$ such that $\epsilon f=\id$, allowing us to apply the above construction with respect to this data.

We equip $D$ with the following $G$-action: $G$ acts on the full subcategories $C$ and $V$ in the obvious way; observe that $V$ is indeed preserved by the $G$-action as its complement is---here we again used that $G$ is a group as opposed to a mere monoid. Finally, if $x\in C$, $y\in V\cap X$ and $\alpha\colon x\to cr(y)$ defines a morphism $x\to y$ in $D$, then we define the $G$-action again via the $G$-action on $C$; note that this indeed makes sense as both $c$ and $r$ are assumed to be $G$-equivariant.

We leave the easy verification that this is indeed a $G$-action and that $j$ and $d$ are $G$-equivariant to the reader.
\end{constr}

As pushouts in $\cat{$\bm G$-Cat}$ are created in $\cat{Cat}$ we immediately get:

\begin{cor}
With respect to the above $G$-action on $D$,
\begin{equation*}
\begin{tikzcd}
A\arrow[d, "c"'] \arrow[r, "i"] & B\arrow["d", d]\\
C \arrow[r, "j"'] & D
\end{tikzcd}
\end{equation*}
becomes a pushout in $\cat{$\bm G$-Cat}$.\qed
\end{cor}

Now we can prove that pushouts along $G$-equivariant Dwyer maps are compatible with passing to fixed points, generalizing \cite[Proposition~2.4]{g-wit}.

\begin{prop}\label{prop:pushout-cat-groups}
Let
\begin{equation}\label{diag:pushout-G-Dwyer}
\begin{tikzcd}
A\arrow[d]\arrow[r, "i"] & B\arrow[d]\\
C\arrow[r] & D
\end{tikzcd}
\end{equation}
be a pushout in $\cat{$\bm G$-Cat}$ such that $i$ is a $G$-equivariant Dwyer map. Then for any subgroup $H\subset G$ also the induced square
\begin{equation}\label{diag:pushout-G-Dwyer-fixed-points}
\begin{tikzcd}
A^H\arrow[d]\arrow[r, "i^H"] & B^H\arrow[d]\\
C^H\arrow[r] & D^H
\end{tikzcd}
\end{equation}
is a pushout (along a Dwyer map).
\begin{proof}
Pick a factorization $i=kf$ together with a right adjoint $r$ of $f$ as in Construction~\ref{constr:pushout-G-Dwyer}; we may then assume without loss of generality that $(\ref{diag:pushout-G-Dwyer})$ is the square from this construction.

The map $i^H$ is a Dwyer map by Lemma~\ref{lemma:Dwyer-closure}-$(\ref{item:Dwyer-fixed-points})$; more precisely, by the proof of the lemma the factorization $i^H=k^Hf^H$ together with the right adjoint $r^H$ and the natural transformations $\eta^H$ and $\epsilon^H$ exhibits $i^H$ as Dwyer map. Thus it suffices to identify $(\ref{diag:pushout-G-Dwyer-fixed-points})$ with the pushout from Construction~\ref{constr:pushout-Dwyer} applied to the ordinary Dwyer map $i^H$ and the above data.

For this we spell out the definitions again: the set of objects of our construction of the pushout $(\ref{diag:pushout-G-Dwyer})$ is $C\amalg V$, where $V$ is the complement of $A$ in $B$, and the $H$-fixed points of this is $C^H\amalg V^H$. From the explicit description $(\ref{eq:hom-pushouts})$ of the Hom-sets we then see that for $x,y\in D^H$
\begin{align*}
\Hom_{D^H}(x,y)&=\begin{cases}
\Hom_C(x,y)^H & \text{if }x,y\in C\\
\Hom_C(x, cr(y))^H & \text{if }x\in C\text{ and }y\in V\cap X\\
\Hom_B(x, y)^H & \text{if }x,y\in V\\
\varnothing^H & \text{otherwise}
\end{cases}\\
&=\begin{cases}
\Hom_{C^H}(x,y) & \text{if }x,y\in C^H\\
\Hom_{C^H}(x, c^Hr^H(y)) & \text{if }x\in C^H\text{ and }y\in V^H\cap X^H\\
\Hom_{B^H}(x, y) & \text{if }x,y\in V^H\\
\varnothing & \text{otherwise}
\end{cases}
\end{align*}
As $V^H$ is the complement of $A^H$ in $B^H$ and $X^H$ is the source of $f^H$, these are literally the objects and morphism sets of the above construction of the pushout of $C^H\gets A^H\to B^H$. Moreover, one checks by direct inspection that the composition is defined in the same way and that also the structure maps $B^H\to D^H$ and $C^H\to D^H$ agree; this finishes the proof.
\end{proof}
\end{prop}

\begin{cor}\label{cor:pushout-dwyer-nerve}
In the situation of the previous corollary, the induced map
\begin{equation*}
\nerve(B)\amalg_{\nerve(A)}\nerve(C)\to\nerve(D)
\end{equation*}
is an $\mathcal F$-weak equivalence for any collection $\mathcal F$ of subgroups of $G$.
\begin{proof}
Let $H\subset G$ be any subgroup. We have to show that the induced map on $H$-fixed points is a weak equivalence. But this map fits into a commutative diagram
\begin{equation*}
\begin{tikzcd}
\nerve(B^H)\amalg_{\nerve(A^H)}\nerve(C^H) \arrow[d]\arrow[r] & \nerve(D^H)\arrow[d]\\
(\nerve B)^H\amalg_{(\nerve A)^H}(\nerve C)^H \arrow[d]\arrow[r] & (\nerve D)^H\arrow[d, equal]\\
\big(\nerve B\amalg_{\nerve A}\nerve C\big)^H\arrow[r] & (\nerve D)^H
\end{tikzcd}
\end{equation*}
where all maps are induced by the relevant universal properties of colimits and limits. The two top vertical arrows are isomorphisms as $\nerve$ is a right adjoint, and the lower left vertical arrow is an isomorphism as fixed points commute with pushouts along monomorphisms in $\cat{Set}$ and hence in $\cat{SSet}$.

But by the previous corollary, the square $(\ref{diag:pushout-G-Dwyer-fixed-points})$ is a pushout along a Dwyer map, hence the top map is a weak equivalence by the classical non-equivariant statement, see Proposition~\ref{prop:pushout-dwyer-nerve}. The claim follows by $2$-out-of-$3$.
\end{proof}
\end{cor}

\subsection{Equivariant Dwyer maps for monoids}
Let $M$ be a monoid in $\cat{Cat}$, i.e.~a small strict monoidal category.

\begin{defi}\label{defi:dwyer-monoid}
An $M$-equivariant functor $i\colon C\to D$ is called an \emph{$M$-equi\-vari\-ant Dwyer map} if it is a $\core(\Ob M)$-equivariant Dwyer map in the sense of Definition~\ref{defi:dwyer-group}, where $\core(\Ob M)$ denotes the maximal subgroup of the discrete monoid $\Ob M$ of objects of $M$.
\end{defi}

Slightly expanding the above definition this means that $i$ is a sieve and that we can find a factorization $i=jf$ \emph{into $\core(\Ob M)$-equivariant functors} with certain properties. However, all of the results below will just follow formally from the corresponding results for equivariant Dwyer maps with respect to groups established above.

\begin{cor}\label{cor:pushout-dwyer-monoid}
Let
\begin{equation*}
\begin{tikzcd}
A\arrow[d] \arrow[r, "i"] & B\arrow[d]\\
C\arrow[r] & D
\end{tikzcd}
\end{equation*}
be a pushout in the category $\cat{$\bm M$-Cat}$ of small categories with $M$-action, and assume that $i$ is an $M$-equivariant Dwyer map. Then for any subgroup $H\subset\Ob M$ the induced square
\begin{equation*}\label{diag:pushout-cat-fixed-points}
\begin{tikzcd}
A^H\arrow[d] \arrow[r, "i^H"] & B^H\arrow[d]\\
C^H\arrow[r] & D^H
\end{tikzcd}
\end{equation*}
is a pushout along a Dwyer map in $\cat{Cat}$.
\begin{proof}
As pushouts in both $\cat{$\bm M$-Cat}$ as well as $\cat{$\bm{\core(\Ob M)}$-Cat}$ are created in $\cat{Cat}$, this is immediate from Proposition~\ref{prop:pushout-cat-groups}.
\end{proof}
\end{cor}

In order to prove the analogue of Corollary~\ref{cor:pushout-dwyer-nerve} we first have to explain how to make $\nerve(C)$ into an $\nerve(M)$-simplicial set for a given $M$-category $C$:

\begin{constr} We lift the adjunction $\h\colon\cat{SSet}\rightleftarrows\cat{Cat} :\!\nerve$ to
\begin{equation*}
\h_{M}\colon\cat{\bm{$\nerve(M)$}-SSet}\rightleftarrows\cat{$\bm{M}$-Cat} :\!\nerve_{M}
\end{equation*}
as follows: on underlying categories or simplicial sets $\h_M$ and $\nerve_M$ agree with $\h$ and $\nerve$, respectively; in particular, this determines their definition on morphisms. If $C$ is an $M$-category, then $\nerve_M(C)=\nerve(C)$ carries the $\nerve(M)$-action given by the composition
\begin{equation*}
\nerve(M)\times\nerve(C)\xrightarrow{\cong}\nerve(M\times C)\xrightarrow{\nerve(\textup{action})}\nerve(C)
\end{equation*}
where the left hand map is the inverse of the canonical isomorphism. Similarly, if $X$ is an $\nerve(M)$-simplicial set, then the $M$-action on $\h_M(X)=\h X$ is given by
\begin{equation*}
M\times\h X\xrightarrow{\cong}\h\nerve(M)\times\h X\xrightarrow{\cong} \h(\nerve(M)\times X)\xrightarrow{\h(\textup{action})}\h X.
\end{equation*}
Here the first map is the inverse of the counit of $\h\dashv\nerve$ (using that $\nerve$ is fully faithful) and the second map is as above (using that $\h$ happens to preserve products).

We omit the easy verification that this is well-defined and that the unit and counit of the original adjunction lift to natural transformations $\id\Rightarrow\nerve_M\h_M$ and $\h_M\nerve_M\Rightarrow\id$, respectively. It then follows formally that these exhibit $\h_M$ as left adjoint of $\nerve_M$. We moreover observe that for any $M$-category $C$ and any $m\in\Ob(M)$ the action map $m.\blank\colon\nerve_M(C)\to\nerve_M(C)$ agrees with $\nerve(m.\blank)$ as a map of simplicial sets.
\end{constr}

With this terminology we can now formulate the desired generalization to categorical monoids:

\begin{cor}\label{cor:pushout-dwyer-monoid-nerve}
Let
\begin{equation*}
\begin{tikzcd}
A\arrow[d] \arrow[r, "i"] & B\arrow[d]\\
C\arrow[r] & D
\end{tikzcd}
\end{equation*}
be a pushout in $\cat{$\bm M$-Cat}$ such that $i$ is an $M$-equivariant Dwyer map. Then the induced map $\nerve_M(B)\amalg_{\nerve_M(A)}\nerve_M(C)\to\nerve_M(D)$ is an $\mathcal F$-weak equivalence for any collection $\mathcal F$ of subgroups of $\nerve(M)_0$.
\begin{proof}
As all the relevant pushouts are created in $\cat{Cat}$ or $\cat{SSet}$, and since the action of any $H\subset\nerve(M)_0$ on $\nerve_M(D)$ is just the one given by functoriality, this follows from Corollary~\ref{cor:pushout-dwyer-nerve}.
\end{proof}
\end{cor}

\begin{cor}\label{cor:homotopy-pushout-dwyer-nerve}
In the situation of the previous corollary, the square
\begin{equation*}
\begin{tikzcd}
\nerve_M(A)\arrow[d] \arrow[r, "\nerve_M(i)"] &[1em] \nerve_M(B)\arrow[d]\\
\nerve_M(C)\arrow[r] & \nerve_M(D)
\end{tikzcd}
\end{equation*}
is a homotopy pushout in the $\mathcal F$-model structure on $\cat{$\bm{\nerve(M)}$-SSet}$ for any collection $\mathcal F$ of finite subgroups of $\nerve(M)_0$.
\begin{proof}
By the previous corollary, the induced map
\begin{equation*}
\nerve_M(B)\amalg_{\nerve_M(A)}\nerve_M(C)\to\nerve_M(D)
\end{equation*}
is a weak equivalence. On the other hand, $i$ is in particular a fully faithful embedding, so that $\nerve_M(i)$ is an underlying cofibration. Proposition~\ref{prop:equiv-model-structure} then implies that the left hand side already represents the homotopy pushout, finishing the proof.
\end{proof}
\end{cor}

Similarly, one generalizes Lemma~\ref{lemma:Dwyer-closure} to all categorical monoids. Below, we will freely refer to Lemma~\ref{lemma:Dwyer-closure} whenever we actually need the corresponding statement for monoid actions.

\subsection{The equivariant \texorpdfstring{$\bm{\Ex}$}{Ex}-functor} In order to construct the equivariant model structure on $\cat{$\bm M$-Cat}$, we will need an equivariant generalization of the usual $\Sd\dashv\Ex$ adjunction, which turns out to be slightly more subtle than in the case of the adjunction $\h\dashv\nerve$, also cf.~\cite[Construction~2.8]{schwede-cat}:

\begin{constr}
Let $N$ be a simplicial monoid. We define $\Ex_N\hskip 0pt minus 1.25pt\colon\hskip 0pt minus .8pt\cat{$\bm N$-SSet}\hskip0pt minus .8pt\to\cat{$\bm N$\hskip 0pt minus 1.2pt-SSet}$ as follows: on underlying simplicial sets, $\Ex_N$ agrees with the usual $\Ex$; in particular, this determines the definition of $\Ex_N$ on morphisms.

If now $X$ is any $N$-simplicial set, then we equip $\Ex_N(X)=\Ex(X)$ with $N$-action given by
\begin{equation*}
N\times\Ex(X)\xrightarrow{e\times\Ex(X)}\Ex(N)\times\Ex(X)\xrightarrow{\cong}\Ex(N\times X)\xrightarrow{\Ex(\textup{action})} \Ex(X);
\end{equation*}
here $e$ is the usual natural transformation $\id\Rightarrow\Ex$ \cite[Section~3]{sd-ex} and the second map is the inverse of the canonical isomorphism. We omit the easy verification that this is well-defined, that the natural transformation $e$ lifts to $e_N\colon\id\Rightarrow\Ex_N$, and that any $n\in N_0$ acts on $\Ex_N(X)=\Ex(X)$ by $\Ex(n.\blank)$.
\end{constr}

\begin{lemma}\label{lemma:ExN-lim-colim}
The functor $\Ex_N$ preserves small limits and filtered colimits.
\begin{proof}
As limits in colimits in $\cat{$\bm N$-SSet}$ are created in $\cat{SSet}$ and as $\Ex_N$ agrees with $\Ex$ on underlying simplicial sets, this is a immediate consequence of the corresponding statement for $\Ex$.
\end{proof}
\end{lemma}

The Special Adjoint Functor Theorem implies:

\begin{cor}
$\Ex_N$ admits a left adjoint $\Sd_N$.\qed
\end{cor}

Let us fix such an adjunction for the rest of this article.

\begin{constr}\label{constr:dN}
Let $X\in\cat{$\bm N$-SSet}$ be arbitary. We define $d_N\colon \Sd_NX\to X$ to be the adjunct of $e_N\colon X\to \Ex_NX$, i.e.~we have commutative diagrams
\begin{equation*}
\begin{tikzcd}
\Sd_NX\arrow[dr, "d_N"', bend right=10pt]\arrow[r, "\Sd_Ne_n"] &[1em] \Sd_N\Ex_NX\arrow[d, "\epsilon"]\\
&X
\end{tikzcd}
\qquad\text{and}\qquad
\begin{tikzcd}
X\arrow[r, "\eta"]\arrow[dr, "e_N"', bend right=10pt] & \Ex_N\Sd_NX\arrow[d, "\Ex_Nd_N"]\\
& \Ex_NX.
\end{tikzcd}
\end{equation*}
Obviously, the $d_N$ assemble into a natural transformation $d_N\colon\Sd_N\Rightarrow\id$, and this is by definition the total mate of the square
\begin{equation}\label{diag:def-dN}
\begin{tikzcd}
\cat{$\bm N$-SSet} & \arrow[l, "\id"']\cat{$\bm N$-SSet}\twocell[dl, "\scriptscriptstyle e_N"{xshift=2.5pt, yshift=-7.5pt}]\\
\cat{$\bm N$-SSet}\arrow[u, "\Ex_N"] & \arrow[l, "\id"]\cat{$\bm N$-SSet}\arrow[u, "\id"']
\end{tikzcd}
\end{equation}
(picking the trivial adjunctions for all the identity arrows).
\end{constr}

We now turn to some properties of the adjunction $\Sd_N\dashv\Ex_N$ as well as the natural transformations $d_N$ and $e_N$:

\begin{lemma}\label{lemma:ExN-homotopical}
\begin{enumerate}
\item The functor $\Ex_N$ preserves $\mathcal F$-weak equivalences for any collection $\mathcal F$ of subgroups of $N_0$.
\item The natural transformation $e_N$ is a levelwise weak equivalence.
\end{enumerate}
\begin{proof}
By $2$-out-of-$3$ it suffices to prove the second statement. For this we have to show that for each $H\subset N_0$ and each $X\in\cat{$\bm N$-SSet}$ the map $(e_N)^H\colon X^H\to\Ex_N(X)^H$ is a weak equivalence. However, we have identified the $H$-action on $\Ex_N(X)=\Ex(X)$ as the one induced by functoriality. As $\Ex$ is a right adjoint, it preserves limits, so we have a canonical isomorphism $\sigma\colon\Ex(X^H)\cong\Ex(X)^H$. Naturality and the universal property of limits then imply that the composition
\begin{equation*}
X^H\xrightarrow{e}\Ex(X^H)\xrightarrow{\sigma} \Ex(X)^H
\end{equation*}
(where the left hand map is the \emph{ordinary} $e\colon\id\Rightarrow\Ex$ evaluated at $X^H$) agrees with $e^H=(e_N)^H$. Hence the claim follows from the fact that the ordinary $e$ is a levelwise weak equivalence \cite[Lemma~7.4]{sd-ex}.
\end{proof}
\end{lemma}

\begin{lemma}\label{lemma:SdN-repr}
Let $H\subset N_0$ be any subgroup. There is a natural isomorphism $\tau$ filling
\begin{equation*}
\begin{tikzcd}
\cat{SSet} \arrow[r, "N/H\times\blank"]\arrow[d, "\Sd"'] &[1em] \cat{$\bm N$-SSet}\arrow[d, "\Sd_N"]\\
\cat{SSet} \twocell[ur, "\scriptscriptstyle \tau"{xshift=2.5pt, yshift=-7pt}] \arrow[r, "N/H\times\blank"'] & \cat{$\bm N$-SSet},
\end{tikzcd}
\end{equation*}
and moreover $\tau$ can be chosen in such a way that for each $K\in\cat{SSet}$ the diagram
\begin{equation*}
\begin{tikzcd}
N/H\times\Sd K \arrow[r, "\tau", "\cong"']\arrow[dr, "N/H\times d"', bend right=10pt] &\Sd_N(N/H\times K)\arrow[d, "d_N"]\\
& N/H\times K
\end{tikzcd}
\end{equation*}
commutes (where $d$ as usual denotes the adjunct of $e$).
\begin{proof}
We recall from the proof of the previous lemma that we have a natural isomorphism filling
\begin{equation*}
\begin{tikzcd}
\cat{SSet} & \arrow[l, "(\blank)^H"'] \cat{$\bm{N}$-SSet}\twocell[dl]\\
\cat{SSet}\arrow[u, "\Ex"] & \arrow[l, "(\blank)^H"]\cat{$\bm{N}$-SSet},\arrow[u, "\Ex_N"']
\end{tikzcd}
\end{equation*}
namely the inverse of the canonical comparison map. We take $\tau$ to be the total mate of this, which is a natural isomorphism
\begin{equation*}
N/H\times\Sd(\blank)\cong\Sd_N(N/H\times\blank).
\end{equation*}
It remains to prove the compatibility of $\tau$ with $d$ and $d_N$. For this we observe that $N/H\times d$ is by definition and the compatibility of mates with pastings the total mate of
\begin{equation*}
\begin{tikzcd}
\cat{SSet} & \arrow[l, "="'] \cat{SSet}\twocell[dl, "\scriptscriptstyle e"{xshift=4pt,yshift=-3.5pt}] & \arrow[l, "(\blank)^H"'] \cat{$\bm N$-SSet}\twocell[dl, "\scriptscriptstyle="{xshift=7pt,yshift=-3.5pt}]\\
\cat{SSet}\arrow[u, "\Ex"] & \arrow[l, "="] \arrow[u, "="'] \cat{SSet} & \arrow[l, "(\blank)^H"]\arrow[u, "="'] \cat{$\bm N$-SSet}.
\end{tikzcd}
\end{equation*}
But on the other hand this pasting agrees by the proof of the previous lemma with
\begin{equation*}
\begin{tikzcd}
\cat{SSet} & \arrow[l, "(\blank)^H"'] \cat{$\bm N$-SSet}\twocell[dl, "\scriptscriptstyle\text{can}^{-1}"{xshift=11pt,yshift=-3.5pt}] & \arrow[l, "="'] \cat{$\bm N$-SSet}\twocell[dl, "\scriptscriptstyle e_N"{xshift=7pt,yshift=-5pt}]\\
\cat{SSet}\arrow[u, "\Ex"] & \arrow[l, "(\blank)^H"] \arrow[u, "\Ex_N"'] \cat{$\bm N$-SSet} & \arrow[l, "="]\arrow[u, "="'] \cat{$\bm N$-SSet}
\end{tikzcd}
\end{equation*}
whose total mate is (again using compatibility of mates with pasting) precisely $d_N\circ\tau$. This finishes the proof.
\end{proof}
\end{lemma}

\begin{cor}\label{cor:SdN-ExN-unit-we}
Let $X\in\cat{$\bm N$-SSet}$ be isomorphic to $N/H\times K$ for some $K\in\cat{SSet}$ and some subgroup $H\subset N_0$. Then:
\begin{enumerate}
\item $d_N\colon \Sd_NX\to X$ is an $\mathcal F$-weak equivalence for any collection $\mathcal F$ of subgroups of $N_0$.
\item $\eta\colon X\to\Ex_N\Sd_NX$ is a weak equivalence.
\end{enumerate}
\begin{proof}
By naturality we may assume without loss of generality that $X$ is actually equal to $N/H\times K$. For the first statement we then simply invoke the previous lemma together with the fact that $d\colon\Sd K\to K$ is an ordinary weak equivalence \cite[Lemma~7.5]{sd-ex}.

For the second statement we consider the commutative diagram
\begin{equation*}
\begin{tikzcd}
X\arrow[r, "\eta"]\arrow[dr, "e_N"', bend right=10pt] & \Ex_N\Sd_NX\arrow[d, "\Ex_Nd_N"]\\
& \Ex_NX
\end{tikzcd}
\end{equation*}
from Construction~\ref{constr:dN}. The first part together with Lemma~\ref{lemma:ExN-homotopical} implies that both $\Ex_Nd_N$ and $e_N$ are weak equivalences; the claim follows by $2$-out-of-$3$.
\end{proof}
\end{cor}

\subsection{Categories vs.~simplicial sets} Let $M$ be a monoid in $\cat{Cat}$ and let $\mathcal F$ be a collection of finite subgroups of $\Ob M$, which we will confuse with subgroups of $(\nerve M)_0$. In order to construct the desired model structure on $\cat{$\bm M$-Cat}$ together with a Quillen equivalence to $\cat{$\bm M$-SSet}$, we will need the following mild technical condition:

\begin{defi}\label{defi:good-subgroup}
A subgroup $H\subset\Ob(M)$ is called \emph{good} if the right $H$-action on $M$ given by right multiplication is free. A collection $\mathcal F$ of subgroups of $\Ob M$ is called \emph{good}, if all $H\in\mathcal F$ are good.
\end{defi}

\begin{ex}
If $\Ob(M)=G$ is a group, then any subgroup $H\subset G$ is good.
\end{ex}

\begin{ex}
Every subgroup of $\Ob(E\mathcal M)=\mathcal M$ is good as injections of sets are monomorphisms. More generally, all subgroups of $\Ob(E\mathcal M\times G)$ are good.
\end{ex}

\begin{thm}\label{thm:cat-vs-sset}
Let $\mathcal F$ be a good collection of finite subgroups of $\Ob M$. Then there exists a unique model structure on $\cat{$\bm M$-Cat}$ such that a map $f\colon C\to D$ is a weak equivalence or fibration if and only if for each $H\in\mathcal F$ the map $f^H\colon C^H\to D^H$ is a weak equivalence or fibration, respectively, in the Thomason model structure.

This model structure is combinatorial with generating cofibrations
\begin{equation*}
\{M/H\times\h\Sd^2\del\Delta^n\hookrightarrow M/H\times\h\Sd^2\Delta^n : n\ge 0,\text{ }H\in\mathcal F\}
\end{equation*}
and generating acyclic cofibrations
\begin{equation*}
\{M/H\times\h\Sd^2\Lambda^n_k\hookrightarrow M/H\times\h\Sd^2\Delta^n : 0\le k\le n,\text{ }H\in\mathcal F\}.
\end{equation*}
Moreover, it is proper with homotopy pushouts and pullbacks created by $\nerve_M$, and filtered colimits in it are homotopical.

Finally, the adjunction
\begin{equation}\label{eq:Quillen-equiv-cat-sset}
\h_M\Sd^2_{\nerve(M)}\colon\cat{$\bm{\nerve(M)}$-SSet}\rightleftarrows\cat{$\bm M$-Cat}:\! \Ex^2_{\nerve(M)}\nerve_M
\end{equation}
is a Quillen equivalence when we equip $\cat{$\bm{\nerve(M)}$-SSet}$ with the $\mathcal F$-model structure (viewing the elements of $\mathcal F$ as subgroups of $\nerve(M)_0$ now).
\end{thm}

If $M=G$ is a discrete group, the above result (without the finiteness condition on $\mathcal F$) was proven by Bohmann, Mazur, Osorno, Ozornova, Ponto, and Yarnall \cite[Theorems~A and~B]{g-wit} although they only explicitly state their result for the collection of \emph{all} subgroups.

\begin{proof}
Recall from Proposition~\ref{prop:equiv-model-structure} that the $\mathcal F$\hskip 0pt minus 1.5pt-model structure on $\cat{$\bm{\nerve(M)}$-SSet}$ is proper, simplicial, that filtered colimits in it are homotopical, and that it is cofibrantly generated with generating cofibrations
\begin{equation}\label{eq:recollection-M-sset-I}
I=\{\nerve(M)/H\times\del\Delta^n\to\nerve(M)/H\times\Delta^n : n\ge 0,\text{ }H\in\mathcal F\}
\end{equation}
and generating acyclic cofibrations
\begin{equation*}
J=\{\nerve(M)/H\times\Lambda^n_k\to\nerve(M)/H\times\Delta^n : 0\le k\le n,\text{ }H\in\mathcal F\}.
\end{equation*}

Let us verify the conditions of Proposition~\ref{prop:transfer-criterion}.
It is clear that $\cat{$\bm M$-Cat}$ is locally presentable. Moreover, each $N(M)/H$ is evidently cofibrant in $\cat{$\bm{\nerve(M)}$-SSet}$, and hence so are the sources of the generating acyclic cofibrations as the model structure is simplicial. To complete the verification of Condition~$(\ref{item:tp-unit-we})$ we will prove more generally that the unit is a weak equivalence for each $X=\nerve(M)/H\times K$ with $H\in\mathcal F$ and any $K\in\cat{SSet}$ that can be equipped with the structure of a simplicial complex, i.e.~that can be embedded into the nerve of a poset. For this we recall that a standard choice of unit is given by the composition
\begin{equation*}
\begin{aligned}
X&\xrightarrow{\eta}\Ex_{\nerve(M)}\Sd_{\nerve(M)} X \xrightarrow{\Ex_{\nerve(M)}\eta\Sd_{\nerve(M)}} \Ex_{\nerve(M)}^2\Sd_{\nerve(M)}^2 X \\
&\qquad\xrightarrow{\Ex_{\nerve(M)}^2\eta\Sd_{\nerve(M)}^2} \Ex_{\nerve(M)}^2\nerve_M\h_M\Sd_{\nerve(M)}^2 X
\end{aligned}
\end{equation*}
where the first two maps come from the unit of $\Sd_{\nerve(M)}\dashv\Ex_{\nerve(M)}$ and the final one is induced by the unit of $\h_{M}\dashv\nerve_M$.

By Corollary~\ref{cor:SdN-ExN-unit-we} the first map is a weak equivalence. As $\Sd_{\nerve(M)}(X)\cong \nerve(M)/H\times\Sd K$  (Lemma~\ref{lemma:SdN-repr}) and as $\Ex_{\nerve(M)}$ is homotopical (Lemma~\ref{lemma:ExN-homotopical}), the corollary also implies that the second map is a weak equivalence. For the final map we use Lemma~\ref{lemma:SdN-repr} twice to see $\Sd_{\nerve(M)}^2X\cong \nerve(M)/H\times\Sd^2K$. Now as an ordinary simplicial set $\nerve(M)/H$ lies in the essential image of $\nerve$ as the nerve preserves \emph{free} quotients, and so does $\Sd^2K$ by \cite[discussion after Proposition~2.5]{thomason-cat} as $K$ was assumed to admit the structure of a simplicial complex. Since $\nerve$ preserves products, we see that the underlying simplicial set of $\Sd_{\nerve(M)}^2X$ indeed lies in the essential image of $\nerve$. But as a map of simplicial sets, the unit $\eta\colon Y\to \nerve_M\h_M(Y)$ agrees with the usual unit $\eta\colon Y\to\nerve\h Y$ for any $Y\in\cat{$\bm M$-SSet}$; as the nerve is fully faithful, we see that this is in fact an isomorphism as soon as the underlying simplicial set of $Y$ lies in the essential image of $\nerve$, finishing the verification of $(\ref{item:tp-unit-we})$.

For Condition~$(\ref{item:tp-pushout})$---i.e.~that the right adjoint sends pushouts along generating cofibrations to homotopy pushouts---we show more generally (cf.~Lem\-ma~\ref{lemma:Dwyer-closure}-$(\ref{item:Dwyer-product})$) that pushouts along $M$-equivariant Dwyer maps are sent to homotopy pushouts by $\Ex_{\nerve(M)}^2\circ\nerve_M$, which is immediate from Corollary~\ref{cor:homotopy-pushout-dwyer-nerve} together with Lemma~\ref{lemma:ExN-homotopical}.

Finally, $\Ex_{\nerve(M)}$ preserves filtered colimits by Lemma~\ref{lemma:ExN-lim-colim}, and the same argument as employed there shows that also $\nerve_M$ preserves filtered colimits. Thus, also the composition $\Ex_{\nerve(M)}^2\nerve_M$ preserves filtered colimits, verifying Condition~$(\ref{item:tp-filtered})$.

Thus, the proposition applies and we see that $\cat{$\bm M$-SSet}$ carries a model structure such that a map $f$ is weak equivalence or fibration if and only if $\Ex_{\nerve(M)}^2\nerve(f)$ is, i.e.~if and only if $(\Ex_{\nerve(M)}^2\nerve_M(f))^H$ is a weak equivalence or fibration in the Kan Quillen model structure for every $H\in\mathcal F$. As both $\Ex_{\nerve(M)}$ and $\nerve_M$ commute with $(\blank)^H$ this is indeed equivalent to the condition stated in the theorem.

Moreover, the proposition tells us that $(\ref{eq:Quillen-equiv-cat-sset})$ is a Quillen equivalence, that the model structure obtained this way is proper with homotopy pushouts and pullbacks created by $\nerve_M$, and that filtered colimits in it are homotopical. Moreover, it shows that the model structure is cofibrantly generated (hence combinatorial) with generating cofibrations $\h_M\Sd^2_{\nerve(M)}(I)$ and generating acyclic cofibrations $\h_M\Sd^2_{\nerve(M)}(J)$. Finally, again using that the nerve preserves free quotients, Lemma~\ref{lemma:SdN-repr} tells us that also the sets from the theorem form generating cofibrations and generating acyclic cofibrations, respectively.
\end{proof}

\begin{cor}\label{cor:homotopy-pushout-m-cat}
In the situation of the theorem, pushouts along $M$-equi\-vari\-ant Dwyer maps are homotopy pushouts.
\begin{proof}
We have seen in the above proof that $\nerve_M$ sends such squares to homotopy pushouts.
\end{proof}
\end{cor}

\begin{cor}
In the above situation, a commutative square is a homotopy pushout if and only if for every $H\in\mathcal F$ the induced square on $H$-fixed points is a homotopy pushout in the Thomason model structure on $\cat{Cat}$.
\begin{proof}
Using that $\nerve_M\colon\cat{$\bm M$-Cat}\to\cat{$\bm{\nerve(M)}$-SSet}$ and $\nerve\colon\cat{Cat}\to\cat{SSet}$ create homotopy pushouts, this follows from the characterization given in Proposition~\ref{prop:equiv-model-structure}.
\end{proof}
\end{cor}

Together with Lemma~\ref{lemma:ExN-homotopical} we get:

\begin{cor}\label{cor:cat-vs-sset-quasi-cat}
For any good family $\mathcal F$, the homotopical functor
\begin{equation*}
\nerve_{M}\colon\cat{$\bm{M}$-Cat}_{\textup{$\mathcal F$-w.e.}}\to\cat{$\bm{\nerve(M)}$-SSet}_{\textup{$\mathcal F$-w.e.}}
\end{equation*}
descends to an equivalence of homotopy categories.\qed
\end{cor}

\subsection{The \texorpdfstring{$\bm G$}{G}-global model structure}
Let us specialize the above to the context of $G$-global homotopy theory:

\begin{cor}
For any discrete group $G$, there is a unique model structure on $\cat{$\bm{E\mathcal M}$-$\bm G$-Cat}$ in which a map $f$ is a weak equivalence or fibration if and only if $f^\phi$ is a weak equivalence or fibration, respectively, in the Thomason model structure for each universal $H\subset\mathcal M$ and each $\phi\colon H\to G$. We call this the \emph{$G$-global model structure} and its weak equivalences the \emph{$G$-global weak equivalences}.

This model category is combinatorial with generating cofibrations
\begin{equation*}\hskip-3pt\hfuzz=3.1pt
\{E\mathcal M\times_\phi G\times(\h\Sd^2\del\Delta^n\hookrightarrow\h\Sd^2\Delta^n) : n\ge 0,\text{ $H\subset\mathcal M$ universal},\phi\colon H\to G\}
\end{equation*}
and generating acyclic cofibrations
\begin{equation*}\hskip-11.63pt\hfuzz=11.63pt
\{E\mathcal M\times_\phi G\times(\h\Sd^2\Lambda^n_k\hookrightarrow\h\Sd^2\Delta^n) : 0\le k\le n,\text{ $H\subset\mathcal M$ universal}, \phi\colon H\to G\}.
\end{equation*}
Moreover, it is proper and filtered colimits in it are homotopical.

Finally, we have a Quillen equivalence
\begin{equation*}
\h_{E\mathcal M\times G}\Sd^2_{E\mathcal M\times G}\colon\cat{$\bm{E\mathcal M}$-$\bm G$-SSet}\rightleftarrows\cat{$\bm{E\mathcal M}$-$\bm G$-Cat}:\! \Ex^2_{E\mathcal M\times G}\nerve_{E\mathcal M\times G}
\end{equation*}
with respect to the model structure from Corollary~\ref{cor:EM-G-SSet-model-structure} on the left hand side.\qed
\end{cor}

\begin{cor}
The homotopical functor
\begin{equation*}
\nerve_{E\mathcal M\times G}\colon\cat{$\bm{E\mathcal M}$-$\bm G$-Cat}_{\textup{$G$-global}}\to\cat{$\bm{E\mathcal M}$-$\bm G$-SSet}_{\textup{$G$-global}}
\end{equation*}
descends to an equivalence of homotopy categories.\qed
\end{cor}

On the other hand, we can equip $\cat{$\bm G$-Cat}$---either by Theorem~\ref{thm:cat-vs-sset} above or by \cite[Theorem~A]{g-wit}---with the equivariant model structure with respect to the collection of all finite subgroups $H\subset G$; we again call this the \emph{proper $G$-equivariant model structure}. As in the simplicial case, the $G$-global and proper $G$-equivariant model structure are related through a chain of four adjoints:

\begin{prop}\label{prop:G-global-vs-G-equiv-cat}
The functor $\triv_{E\mathcal M}\colon\cat{$\bm G$-Cat}_{\textup{proper}}\to\cat{$\bm{E\mathcal M}$-$\bm G$-Cat}_{\textup{$G$-global}}$ is homotopical and the induced functor $\Ho(\triv_{E\mathcal M})$ on homotopy categories fits into a sequence of four adjoints suggestively denoted by
\begin{equation*}
\cat{L}(\blank/E\mathcal M)\dashv\Ho(\triv_{E\mathcal M})\dashv(\blank)^{\cat{R}E\mathcal M}\dashv\mathcal R.
\end{equation*}
Moreover, $(\blank)^{\cat{R}E\mathcal M}$ is a (Bousfield) localization at the $\mathcal E$-weak equivalences.
\begin{proof}
The diagram of homotopical functors
\begin{equation*}
\begin{tikzcd}[column sep=large]
\cat{$\bm G$-Cat}_{\textup{proper}}\arrow[r, "\triv_{E\mathcal M}"]\arrow[d, "\nerve_G"'] & \cat{$\bm{E\mathcal M}$-$\bm G$-Cat}_{\textup{$G$-global}}\arrow[d, "\nerve_{E\mathcal M\times G}"]\\
\cat{$\bm G$-SSet}_{\textup{proper}}\arrow[r, "\triv_{E\mathcal M}"'] & \cat{$\bm{E\mathcal M}$-$\bm G$-SSet}_{\textup{$G$-global}}
\end{tikzcd}
\end{equation*}
commutes by direct inspection, and the vertical maps induce equivalences of homotopy categories by Corollary~\ref{cor:cat-vs-sset-quasi-cat}. Thus, the claim follows from Theorem~\ref{thm:G-global-vs-G-equiv-sset} by a straight-forward diagram chase.
\end{proof}
\end{prop}

\begin{rk}
Analogously to the case of the simplicial models treated in \cite[Theorem~1.2.92]{g-global}, the functor $\triv_{\hskip-.8ptE\hskip-.8pt\mathcal M}$ is easily seen to be right Quillen with respect to the above model structures, so that its left adjoint $(\blank)/E\mathcal M$ is left Quillen, justifying the above notation $\cat{L}(\blank/E\mathcal M)$. Moreover, while $\triv_{E\mathcal M}$ is \emph{not} left Quillen, we can make it into a left Quillen functor by suitably enlarging the generating cofibrations of the $G$-global model structure, e.g.~using \cite[Corollary~A.2.17]{g-global}. With respect to this model structure, $(\blank)^{E\mathcal M}$ is then right Quillen and its right derived functor is then right adjoint to $\Ho(\triv_{E\mathcal M})$ as before.
\end{rk}

\section{\texorpdfstring{$\bm G$}{G}-categories as models of \texorpdfstring{$\bm G$}{G}-global homotopy types}\label{sec:g-cat}
Fix a discrete group $G$. In this final section, we will prove as our main result that already categories with a mere $G$-action model unstable $G$-global homotopy theory. For this we will use the following notion from \cite[Definition~6.1]{sym-mon-global} which generalizes Schwede's \emph{global equivalences} \cite[Definition~3.2]{schwede-cat}:

\begin{defi}
Let $C$ be a $G$-category, let $H$ be a finite group, and let $\phi\colon H\to G$ be a group homomorphism. We define the \emph{$\phi$-`homotopy' fixed points} of $C$ as
\begin{equation*}
C^{\myh\phi}\mathrel{:=}\Fun(EH,\phi^*C)^H
\end{equation*}
where the $H$-action on the right is the diagonal of the $H$-action on $C$ via $\phi$ and the one induced by the right regular action on $EH$.

If $f$ is a $G$-equivariant functor $f\colon C\to D$, then $f^{\myh\phi}$ is defined analogously. We call $f$ a \emph{$G$-global weak equivalence} if $f^{\myh\phi}$ is a weak homotopy equivalence (i.e.~weak equivalence in the Thomason model structure) for every such $\phi$.
\end{defi}

\begin{rk}\label{rk:homotopy-fixed-points}
The above represents the $H$-homotopy fixed points of $\phi^*C$ \emph{with respect to the canonical model structure on $\cat{Cat}$}, i.e.~where the weak equivalences are the equivalences of categories. However, as we are generalizing Thomason's model structure (for which homotopy fixed points look quite different), we have decided to put `homotopy' in quotation marks everywhere.

It is of course crucial that we're taking homotopy fixed points with respect to the `wrong' model structure here as otherwise being a $G$-global weak equivalence would be equivalent to being an underlying weak homotopy equivalence.
\end{rk}

\begin{thm}\label{thm:homotopy-fixed-vs-really-fixed}
There is a unique cofibrantly generated model structure on $\cat{$\bm G$-Cat}$ with weak equivalences the $G$-global weak equivalences and with generating cofibrations given by
\begin{align*}
\{(E\mathcal M\times_\phi G)\times(\h\Sd^2\del\Delta^n\hookrightarrow\h\Sd^2\Delta^n) : {}&n\ge 0,H\subset\mathcal M\text{ universal},\\ &\phi\colon H\to G\text{ homomorphism}\}.
\end{align*}
This model structure is combinatorial, proper, and filtered colimits in it are homotopical. We call it the \emph{thick $G$-global model structure}. A set of generating acyclic cofibrations is given by
\begin{align*}
\{(E\mathcal M\times_\phi G)\times(\h\Sd^2\Lambda^n_k\hookrightarrow\h\Sd^2\Delta^n) : {}&0\le k\le n,H\subset\mathcal M\text{ universal},\\ &\phi\colon H\to G\text{ homomorphism}\}.
\end{align*}
Moreover, we have a Quillen equivalence
\begin{equation}\label{eq:G-forget}
\forget\colon\cat{$\bm{E\mathcal M}$-$\bm G$-Cat}_{\textup{$G$-global}}\rightleftarrows\cat{$\bm G$-Cat}_{\textup{thick $G$-global}} :\!\Fun(E\mathcal M,\blank)
\end{equation}
(where $\Fun(E\mathcal M,C)$ is equipped for every $C\in\cat{$\bm G$-Cat}$ with the left $E\mathcal M$-action induced by the usual right $E\mathcal M$-action on itself), and the right adjoint creates homotopy pushouts and homotopy pullbacks.
\end{thm}

The proof will be given below after some preparations.

\begin{lemma}\label{lemma:hphi-vs-em-phi}
Let $f\colon C\to D$ be a map of $G$-categories, let $H\subset\mathcal M$ be any subgroup, and let $\phi\colon H\to G$. Then $f^{\myh\phi}$ is a weak equivalence in the Thomason model structure if and only if $\Fun(E\mathcal M, f)^\phi$ is so. In particular, $f$ is a $G$-global weak equivalence in $\cat{$\bm G$-Cat}$ if and only if $\Fun(E\mathcal M,f)$ is a $G$-global weak equivalence in $\cat{$\bm{E\mathcal M}$-$\bm G$-Cat}$.
\begin{proof}
This appears as \cite[Remark~4.1.28]{g-global}; as we will need similar arguments below, we spell out the proof in detail.

It will be enough to show that the inclusion $H\hookrightarrow\mathcal M$ induces a $(G\times H)$-equivariant weak equivalence $\Fun(E\mathcal M,C)\to\Fun(EH,C)$ for every $G$-category $C$. For this we observe that there exists a right $H$-equivariant map $r\colon\mathcal M\to H$ as $H$ acts freely from the right on $\mathcal M$. There are then unique isomorphisms $E(i)E(r)=E(ir)\cong\id$ and $E(r)E(i)=E(ri)\cong\id$, and these are automatically $H$-equivariant. It is then clear that the corresponding restrictions exhibit $\Fun(Er,C)$ as a $(G\times H)$-equivariant quasi-inverse to the map in question; in particular, both are $(G\times H)$-equivariant weak equivalences.
\end{proof}
\end{lemma}

\begin{defi}
A small $E\mathcal M$-$G$-category $C$ is called \emph{saturated} if the unit $\eta\colon C\to\Fun(E\mathcal M,\forget C)$ induces equivalences on $\phi$-fixed points for all universal $H\subset\mathcal M$ and all $\phi\colon H\to G$.
\end{defi}

\begin{rk}
As observed in \cite[Remark~4.1.24]{g-global}, an $E\mathcal M$-$G$-category $C$ is saturated if and only if the inclusion $C\hookrightarrow\Fun(E\mathcal M,C)$ of constant diagrams induces equivalences on $\phi$-fixed points for all $\phi$ as above, where the right hand side is now equipped with the diagonal $E\mathcal M$-action. This alternative definition is used in \cite{schwede-k-theory,g-global}.
\end{rk}

\begin{prop}\label{prop:generators-saturated}
Let $K\subset\mathcal M$ be any finite subgroup, let $\psi\colon K\to G$, and let $P$ be any poset (viewed as an $E\mathcal M$-$G$-category with trivial actions). Then $E\mathcal M\times_\psi G\times P$ is saturated.
\begin{proof}
As $\Fun(E\mathcal M,\blank)$, the forgetful functor, and $(\blank)^\phi$ each preserve products, it suffices to show that both $E\mathcal M\times_\psi G$ and $P$ are saturated. But indeed, the first statement is a special case of \cite[Lemma~4.2.10]{g-global}, and for the second statement we observe that $\eta$ is even an isomorphism as $E\mathcal M$ is a connected groupoid while $P$ has no non-trivial isomorphisms.
\end{proof}
\end{prop}

\begin{proof}[Proof of Theorem~\ref{thm:homotopy-fixed-vs-really-fixed}]
Obviously, there is at most one such model structure. We will now verify the assumptions $(\ref{item:tp-unit-we})$, $(\ref{item:tp-pushout})$, and $(\ref{item:tp-filtered})$ of Proposition~\ref{prop:transfer-criterion} for the adjunction $(\ref{eq:G-forget})$. For this we begin by observing that the images of the standard generating (acyclic) cofibrations of $\cat{$\bm{E\mathcal M}$-$\bm G$-Cat}$ are indeed precisely the above sets. Moreover, Lemma~\ref{lemma:hphi-vs-em-phi} tells us that the transferred weak equivalences agree with the $G$-global weak equivalences.

The standard generating acyclic cofibrations of $\cat{$\bm{E\mathcal M}$-$\bm{G}$-Cat}$ have cofibrant sources as the ones of $\cat{$\bm{E\mathcal M}$-$\bm G$-SSet}$ have. The remainder of Condition~$(\ref{item:tp-unit-we})$, stating that the unit is a weak equivalence on $\varnothing$ as well as on sources and targets of the standard generating cofibrations, is a special case of Proposition~\ref{prop:generators-saturated}.

For Condition~$(\ref{item:tp-pushout})$ we consider any pushout
\begin{equation}\label{diag:pushout-g-cat}
\begin{tikzcd}
A \arrow[d]\arrow[r, "i"] & B\arrow[d]\\
C\arrow[r] & D
\end{tikzcd}
\end{equation}
in $\cat{$\bm G$-Cat}$ such that $i$ is a $G$-equivariant Dwyer map; for example, $i$ could be one of the standard generating cofibrations. As all colimits in question are created in $\cat{Cat}$, Lemma~\ref{lemma:pushout-dwyer-fun} shows that applying $\Fun(E\mathcal M,\blank)$ to $(\ref{diag:pushout-g-cat})$ yields a pushout in $\cat{$\bm{E\mathcal M}$-$\bm G$-Cat}$. Moreover, Lemma~\ref{lemma:Dwyer-closure}-$(\ref{item:Dwyer-fun})$ shows that $\Fun(E\mathcal M,i)$ is an $(E\mathcal M\times G)$-equivariant Dwyer map. Thus, Corollary~\ref{cor:homotopy-pushout-m-cat} implies that $\Fun(E\mathcal M,\blank)$ sends $(\ref{diag:pushout-g-cat})$ to a homotopy pushout.

It remains to check that $\Fun(E\mathcal M,\blank)$ preserves filtered colimits up to weak equivalence. For this we observe that the functor $\Fun(E\mathcal M,\blank)^\phi$ is weakly equivalent to $(\blank)^{\myh\phi}=\Fun(EH,\blank)^\phi\colon\cat{$\bm G$-Cat}\to\cat{Cat}$ by the proof of Lemma~\ref{lemma:hphi-vs-em-phi}, so it suffices that $(\blank)^{\myh\phi}$ preserves filtered colimits up to weak equivalence. But as $EH$ is a finite category and since filtered colimits in $\cat{Cat}$ commute with finite limits, it even preserves filtered colimits up to isomorphism.

Hence \hskip0pt minus .2pt Proposition~\hskip0pt minus .2pt \ref{prop:transfer-criterion} \hskip0pt minus .2pt implies \hskip0pt minus .2pt the \hskip0pt minus .2pt existence \hskip0pt minus .2pt of \hskip0pt minus .2pt the \hskip0pt minus .2pt model \hskip0pt minus .2pt structure, \hskip0pt minus .2pt shows that $(\ref{eq:G-forget})$ is a Quillen equivalence, and proves all the desired properties.
\end{proof}

We can also construct a variant of the above model structure with fewer cofibrations, which for $G=1$ recovers Schwede's global model structure on $\cat{Cat}$ recalled in Theorem~\ref{thm:thomason-schwede}:

\begin{prop}\label{prop:thin-g-global}
There is a unique model structure on $\cat{$\bm G$-Cat}$ in which a map is a weak equivalence or fibration if and only if $f^{\myh\phi}$ is a weak equivalence or fibration, respectively, in the Thomason model structure for every universal $H\subset\mathcal M$ and each $\phi\colon H\to G$. We call this the \emph{$G$-global model structure}. It is proper (with homotopy pushouts and pullbacks created by the homotopy fixed point functors $(\blank)^{\myh\phi}$ for varying $\phi$) and combinatorial with generating cofibrations
\begin{equation}\label{eq:g-cat-gen-cof}
\{ EH\times_\phi G\times (\h\Sd^2\del\Delta^n\hookrightarrow\h\Sd^2\Delta^n) : H\text{ finite group},\phi\colon H\to G, n\ge0\}
\end{equation}
and generating acyclic cofibrations
\begin{equation}\label{eq:g-cat-gen-acof}
\{ EH\times_\phi G\times (\h\Sd^2\Lambda_k^n\hookrightarrow\h\Sd^2\Delta^n) : H\text{ finite group},\phi\colon H\to G, 0\le k\le n\}.
\end{equation}
Moreover, filtered colimits in this model structure are homotopcial.

Finally, the adjunction
\begin{equation*}
\id\colon\cat{$\bm G$-Cat}_{\textup{$G$-global}}\rightleftarrows\cat{$\bm G$-Cat}_{\textup{$G$-global thick}} :\!\id
\end{equation*}
is a Quillen equivalence.
\begin{proof}
To construct the model structure and to prove that it has all the properties stated above it suffices to verify the assumptions $(\ref{item:tp-weak-equivalence})$--$(\ref{item:tp-filtered})$ of Proposition~\ref{prop:transfer-criterion} for the adjunction
\begin{equation*}
L\colon\prod_{\phi\colon H\to G}\cat{Cat}\rightleftarrows\cat{$\bm G$-Cat} : \big((\blank)^{\myh\phi}\big)_\phi
\end{equation*}
where the product runs over a set of representatives of isomorphism classes of finite groups $H$ and all homomorphisms $\phi\colon H\to G$. The left adjoint can be calculated as before on objects by $(X_\phi)_\phi\mapsto\coprod_\phi EH\times_\phi G\times X_\phi$ and likewise on morphisms. In particular, the above sets agree up to isomorphism with the images of the standard generating (acyclic) cofibrations of $\prod_{\phi}\cat{Cat}$ under $L$.

It is clear that the right adjoint sends the maps in $(\ref{eq:g-cat-gen-acof})$ to weak equivalences. Moreover, the maps in $(\ref{eq:g-cat-gen-cof})$ are $G$-equivariant Dwyer maps, so the right adjoint sends pushouts along them to homotopy pushouts by the same argument as in the proof of Theorem~\ref{thm:homotopy-fixed-vs-really-fixed}. Finally, the right adjoint clearly preserves filtered colimits.

It is clear that the identity functor $\cat{$\bm G$-Cat}_{\textup{$G$-global}}\to\cat{$\bm G$-Cat}_{\textup{$G$-global thick}}$ preserves and reflects weak equivalences, so it only remains to show that it sends generating cofibrations to cofibrations. To this end we fix a finite group $H$ with a homomorphism $\phi\colon H\to G$, and we pick an injective homomorphism $\iota\colon H\to\mathcal M$ with universal image. To finish the proof, it is now enough to show that $EH\times_\phi G$ is a retract of $E\mathcal M\times_{\phi\iota^{-1}}G$, for which it suffices that $H$ is a right $H$-equivariant retract of $\mathcal M$, where $H$ acts on $\mathcal M$ via $\iota$. But indeed, $\iota\colon H\to\mathcal M$ is an $H$-equivariant injection and $\mathcal M$ is free, so $\iota$ admits an $H$-equivariant retraction as desired.
\end{proof}
\end{prop}

Finally, let us compare the above to the usual proper $G$-equivariant model structure on $\cat{$\bm G$-Cat}$:

\begin{defi}
We call a map $f\colon C\to D$ in $\cat{$\bm G$-Cat}$ a \emph{$G$-equivariant `homotopy' weak equivalence} if $f^{\myh H}=\Fun(EH,f)^H$ is a weak equivalence for every finite subgroup $H\subset G$.
\end{defi}

\begin{thm}\label{thm:g-global-cat-vs-g-proper-cat}
The functor
\begin{equation}\label{eq:Fun-EG-loc}
\Fun(EG,\blank)\colon\cat{$\bm G$-Cat}_{\textup{$G$-global}}\to\cat{$\bm G$-Cat}_{\textup{proper}}
\end{equation}
is homotopical. The induced functor of homotopy categories is is a (Bousfield) localization at the $G$-equivariant `homotopy' weak equivalences, and it is the third term in a sequence of four adjoints.
\begin{proof}
We claim that the diagram
\begin{equation}\label{diag:EG-vs-EM}
\begin{tikzcd}
\cat{$\bm G$-Cat}_{\textup{$G$-global}}\arrow[r, "{\Fun(EG,\blank)}"]\arrow[d, "{\Fun(E\mathcal M,\blank)}"'] & \cat{$\bm G$-Cat}_{\textup{proper}}\arrow[d, "\triv_{E\mathcal M}"]\\
\cat{$\bm{E\mathcal M}$-$\bm G$-Cat}_{\textup{$G$-global}}\arrow[r, "\textup{id}"'] & \cat{$\bm{E\mathcal M}$-$\bm G$-Cat}_{\textup{$\mathcal E$-weak equivalences}}
\end{tikzcd}
\end{equation}
of homotopical functors commutes up to a zig-zag of levelwise weak equivalences. Before we prove this, let us show how it implies the theorem: the vertical maps induce equivalences of homotopy categories by Theorem~\ref{thm:homotopy-fixed-vs-really-fixed} and Proposition~\ref{prop:G-global-vs-G-equiv-cat}, respectively. Thus, the top arrow induces a localization at those maps that are inverted by the functor induced by the lower composite. As the $\mathcal E$-weak equivalences are saturated (being the weak equivalences of a model structure), these are precisely those maps $f$ such that $\Fun(E\mathcal M,f)$ is an $\mathcal E$-weak equivalence, which we as before identify with the $G$-equivariant `homotopy' weak equivalences. Finally, $(\blank)^{\cat{R}E\mathcal M}$ is a localization of $\Ho(\cat{$\bm{E\mathcal M}$-$\bm G$-Cat}_{\textup{$G$-global}})$ at the $\mathcal E$-weak equivalences by Proposition~\ref{prop:G-global-vs-G-equiv-cat}, hence equivalent to the functor induced by
\begin{equation}\label{eq:identity-G-global-vs-E}
\id\colon\cat{$\bm{E\mathcal M}$-$\bm G$-Cat}_{\textup{$G$-global}}\to\cat{$\bm{E\mathcal M}$-$\bm G$-Cat}_{\textup{$\mathcal E$-weak equivalences}}.
\end{equation}
As $(\blank)^{\cat{R}E\mathcal M}$ is the third term in a sequence of four adjoints by the aforementioned proposition, so is the functor induced by $(\ref{eq:identity-G-global-vs-E})$, and hence also the one induced by $\Fun(EG,\blank)$ as desired.

It remains to construct a zig-zag of natural levelwise weak equivalences filling $(\ref{diag:EG-vs-EM})$. For this we will show more generally that for any $E\mathcal M$-$G$-category $C$ the maps
\begin{equation*}
\Fun(EG,C)\to\Fun(EG\times E\mathcal M,C)\gets\Fun(E\mathcal M,C)
\end{equation*}
induced by the projections $EG\gets EG\times E\mathcal M\to E\mathcal M$ are $\mathcal E$-weak equivalences, for which it is enough that the projections are $H$-equivariant equivalences for every universal $H\subset\mathcal M$ and every \emph{injective} $\phi\colon H\to G$ (when we let $H$ act on $G$ via $\phi$). This follows as before as both $G$ and $\mathcal M$ are free right $H$-sets with respect to these actions.
\end{proof}
\end{thm}

\begin{rk}
Using a similar argument as in the proof of Proposition~\ref{prop:thin-g-global} it is not hard to show that $(\ref{eq:Fun-EG-loc})$ is right Quillen for the thick $G$-global model structure whenever the cardinality of $G$ is at most $|\mathcal M|$, so that we have a Quillen adjunction
\begin{equation*}
EG\times\blank\colon \cat{$\bm G$-Cat}_{\textup{proper}}\rightleftarrows\cat{$\bm G$-Cat}_{\textup{thick $G$-global}}: \!\Fun(EG,\blank).
\end{equation*}
\end{rk}

\begin{rk}
Already an ordinary global space has an \emph{underlying $G$-space} for every finite group. In \cite[Example~3.21]{schwede-cat}, Schwede gives an explicit model of this in terms of a Quillen adjunction
\begin{equation}\label{eq:cat-vs-g-cat}
EG\times_G\blank\colon\cat{$\bm G$-Cat}_{\textup{$G$-equivariant}}\rightleftarrows\cat{Cat}_{\mathcal G}:\!\Fun(EG,\blank)
\end{equation}
with homotopical right adjoint for a certain model structure on $\cat{Cat}$ with the same weak equivalences as the global one, but more cofibrations.

The induced adjunction on homotopy categories is in fact a shadow of Theorem~\ref{thm:g-global-cat-vs-g-proper-cat}: for every $\alpha\colon G\to H$, the functor $\alpha^*\colon\cat{$\bm H$-Cat}\to\cat{$\bm G$-Cat}$ is easily seen to be right Quillen with respect to the thick $H$-global and thick $G$-global model structure, respectively. Specializing this to $H=1$ we obtain together with the previous remark a chain of Quillen adjunctions
\begin{equation*}
\begin{tikzcd}
\cat{$\bm G$-Cat}_{\textup{$G$-equivariant}}\arrow[r, "EG\times\blank", shift left=2pt] &[2em] \arrow[l, "{\Fun(EG,\blank)}", shift left=2pt]\cat{$\bm G$-Cat}_{\textup{thick $G$-global}}\arrow[r, "(\blank)/G", shift left=2pt] &[.5em]\arrow[l, "\triv_G", shift left=2pt]\cat{Cat}_{\textup{thick global}}
\end{tikzcd}
\end{equation*}
with homotopical right adjoints, whose composition agrees with $(\ref{eq:cat-vs-g-cat})$.
\end{rk}

As an application of the above comparison, we can now introduce a new model structure on $\cat{$\bm G$-Cat}$ representing proper $G$-equivariant homotopy theory, whose weak equivalences are tested on `homotopy' fixed points:

\begin{thm}
There is a unique model structure on $\cat{$\bm G$-Cat}$ in which a map $f$ is a weak equivalence or fibration if and only if $\Fun(EG,f)^H$ is a weak equivalence or fibration, respectively, in the Thomason model structure on $\cat{Cat}$ for each finite subgroup $H\subset G$. We call this the \emph{thick $G$-equivariant `homotopy' fixed point model structure}; its weak equivalences are precisely the $G$-equivariant `homotopy' weak equivalences. This model structure is combinatorial with generating cofibrations
\begin{equation*}
\{EG\times_HG\times(\h\Sd^2\del\Delta^n\hookrightarrow\h\Sd^2\Delta^n) : H\subset G\textup{ finite},n\ge0\}
\end{equation*}
and generating acyclic cofibrations
\begin{equation*}
\{EG\times_HG\times(\h\Sd^2\Lambda_k^n\hookrightarrow\h\Sd^2\Delta^n) : H\subset G\textup{ finite},0\le k\le n\}
\end{equation*}
and proper; a commutative square is a homotopy pushout or pullback if and only if the induced square on $H$-`homotopy' fixed points is a homotopy pushout or pullback, respectively, in the Thomason model structure on $\cat{Cat}$ for each finite $H\subset G$. Moreover, filtered colimits in this model structure are homotopical.

Likewise, there is a \emph{$G$-equivariant `homotopy' fixed point model structure} on $\cat{$\bm G$-Cat}$ in which a map $f$ is a weak equivalence or fibration if and only if $f^{\myh H}=\Fun(EH,f)^H$ is a weak equivalence or fibration, respectively, in the Thomason model structure for every finite $H\subset G$. It is again proper with the above characterization of homotopy pushouts and pullbacks, and it is morover combinatorial with generating cofibrations
\begin{equation*}
\{EH\times_HG\times(\h\Sd^2\del\Delta^n\hookrightarrow\h\Sd^2\Delta^n) : H\subset G\textup{ finite},n\ge0\}
\end{equation*}
and generating acyclic cofibrations
\begin{equation*}
\{EH\times_HG\times(\h\Sd^2\Lambda_k^n\hookrightarrow\h\Sd^2\Delta^n) : H\subset G\textup{ finite},0\le k\le n\}.
\end{equation*}
Finally, the adjunctions
\begin{align}\label{eq:thick-vs-thin-g-homotopy}
\id\colon\cat{$\bm G$-Cat}_{\textup{$G$-`homotopy'}}&\rightleftarrows\cat{$\bm G$-Cat}_{\textup{thick $G$-`homotopy'}} :\!\id\\
\label{eq:equiv-vs-homotopy}
EG\times\blank\colon\cat{$\bm G$-Cat}_{\textup{proper $G$-equivariant}} &\rightleftarrows \cat{$\bm G$-Cat}_{\textup{thick $G$-`homotopy'}} :\!\Fun(EG,\blank)
\end{align}
are Quillen equivalences.
\begin{proof}
To construct the model structures and to establish the above properties, it is enough to verify the assumptions $(\ref{item:tp-weak-equivalence})$--$(\ref{item:tp-filtered})$ of Proposition~\ref{prop:transfer-criterion}, which can be done just as in the proof of Proposition~\ref{prop:thin-g-global}. The same argument as in Lemma~\ref{lemma:hphi-vs-em-phi} and the aforementioned proposition then shows that $(\ref{eq:thick-vs-thin-g-homotopy})$ is a Quillen adjunction and that both sides have the same weak equivalences, so that it is even a Quillen equivalence.

For the final statement we observe that $\Fun(EG,\blank)$ preserves (and reflects) weak equivalences as well as fibrations by definition; in particular, $(\ref{eq:equiv-vs-homotopy})$ is a Quillen adjunction. It only remains to show that $\Fun(EG,\blank)$ descends to an equivalence, which is immediate from Theorem~\ref{thm:g-global-cat-vs-g-proper-cat}.
\end{proof}
\end{thm}

\frenchspacing
\newcommand{\etalchar}[1]{$^{#1}$}

\parindent=0pt
Tobias Lenz\\
Max-Planck-Institut für Mathematik\\
Vivatsgasse 7\\
53111 Bonn (Germany)\par\medskip

\textit{Current address:}\\
Mathematical Institute, University of Utrecht\\
Budapestlaan 6\\
3584$\,$CD Utrecht (The Netherlands)\\
\texttt{t.lenz@uu.nl}
\end{document}